\documentclass[a4paper, 11pt]{article}
\usepackage{mystyle}

\begin{document}
	\onehalfspacing
	
%-------------- Information For The Title Page
% Title page info (see uowthesistitlepage package)
	\title{Hierarchies of holonomy groupoids for foliated bundles} 
	\author{Lachlan E. MacDonald\\School of Mathematical Sciences\\
		The University of Adelaide\\
		Adelaide, SA, 5000}

	\date{November 2020}
	
	\maketitle
	
	\begin{abstract}
		We give a new construction of the holonomy groupoid of a regular foliation in terms of a partial connection on a diffeological principal bundle of germs of transverse parametrisations, which may be viewed as a gauge-theoretic systematisation of Winkelnkemper's original construction.  We extend these ideas to construct a novel holonomy groupoid for any foliated bundle, which we prove sits at the top of a hierarchy of diffeological jet holonomy groupoids associated to the foliated bundle.  This shows that while the Winkelnkemper holonomy groupoid is the smallest Lie groupoid that integrates a foliation, it is far from the smallest diffeological groupoid that does so.
	\end{abstract}
	
	\tableofcontents

\section{Introduction}

The holonomy groupoid of a regular foliation, introduced by Winkelnkemper \cite{wink}, is a foundational object for the study of foliations via noncommutative geometry \cite{ncg}.  Despite its name, Winkelnkemper's construction of the holonomy groupoid bears no clear relation to the gauge-theoretic notion of holonomy obtained via bundles and connections.  In this paper we introduce a diffeological gauge theory for foliations that enables a new construction of Winkelnkemper's holonomy groupoid in terms of fibre bundles and connections, and which generalises easily to give novel holonomy groupoids for any foliated bundle, which fit into a natural hierarchy.

Holonomy is thought of in the language of differential geometry as an equivalence relation between smooth loops that is defined in terms of a certain first order differential equation.  More precisely, suppose we are given a fibre bundle $\pi_{B}:B\rightarrow M$ and a connection $\omega\in T^{*}B\otimes VB$, which we think of as a vector bundle projection from the tangent bundle $TB$ onto the vertical bundle $VB:=\ker(d\pi_{B})$  of $B$. To any smooth loop $\gamma$ based at $x\in M$ we have the differential equation
\begin{equation}\label{holeqn}
\omega\big(\partial_{t}\tilde{\gamma}\big) = 0 
\end{equation}
subject to the constraint $\pi_{B}(\tilde{\gamma}) = \gamma$, whose solutions $\tilde{\gamma}$ define a foliation of the restriction of $B$ to the image of $\gamma$.  One thereby obtains a diffeomorphism $T(\gamma)_{x}$ of $B_{x}$, called the \emph{parallel transport} of $\gamma$, which is defined by sending any $b\in B_{x}$ to the endpoint of the solution $\tilde{\gamma}$ through $b$.  The group of all such parallel transport diffeomorphisms is called the \emph{holonomy group} of $\omega$ at $x$ - if in particular $\pi_{B}$ is a principal $G$-bundle for some Lie group $G$, then each $T(\gamma)_{x}$ is an element of $G$.  There is of course nothing stopping one from generalising this idea to paths which are not loops - in this way one envisages a \emph{holonomy groupoid}, consisting of equivalence classes of paths whose parallel transport diffeomorphisms are equal.

From any foliated manifold $(M,\FF)$ there arises a canonical diffeological groupoid $\HH(M,\FF)$ \cite{wink, holimper, debord, iakovos2, iakovos4, VilGar1}, frequently called the \emph{holonomy groupoid} of $(M,\FF)$, whose constructions so far have appeared to take on a different character in that they tend not to invoke bundles or connections.  In this paper, we unify the two notions of holonomy using diffeological fibre bundles and partial connections thereon in the case of \emph{regular} foliated manifolds, by which we mean foliations whose leaves all have the same dimension.  This unification is obtained by passing from the category of smooth fibre bundles over smooth manifolds to diffeological fibre bundles over diffeological spaces.  The extra flexibility afforded by diffeology, which appears to be the most natural category for many constructions relating to foliations \cite{HMVSC, VilGar1}, permits the construction of novel diffeological fibre bundles consisting of germs of smooth functions.  Such bundles play a crucial role in the theory we develop in this paper, and we believe they may be of independent utility and interest.

Let us outline the contents of the paper.  In the first section we assemble the necessary background from the theory of regular foliated manifolds, including a recollection of Winkelnkemper's construction of the holonomy groupoid.  In this section we also recall Kamber and Tondeur's definition of foliated bundles, and give an account of the construction of the associated transverse jet bundles.  \begin{comment}Despite jet bundle theory having been well-known in the non-foliated context for decades, we have been unable to locate any literature which deals with jet bundles in a foliated context.  For this reason our exposition on this topic includes many details.\end{comment}

The second section presents a quick account of the background needed from diffeology.  Much of this material is sourced from the marvellous book \cite{diffeology} and the papers \cite{hector, hector2, cw}.  We do, however, introduce some new adaptions of old ideas to the diffeological context.  Precisely, we generalise Kamber and Tondeur's concept of a partial connection in a fibre bundle to arbitrary diffeological fibre bundles.  We also introduce a diffeological version of the Moore path category $\PP(X)$ for any diffeological space $X$, which is a diffeological category whose morphisms are pairs $(\gamma,d)$ consisting of $d\in\RB_{+}:=[0,\infty)$ and a smooth path $\gamma:\RB_{+}\rightarrow X$ which is constant on $[d,\infty)$.  Composition of morphisms in this category is just concatenation of paths, which we show is smooth.  Using the theory of diffeological tangent bundles from \cite{hector, hector2, cw} we also specialise these diffeological Moore path categories to objects which consist only of paths whose tangents lie in some specified distribution, as is frequently required for foliations.  Precisely, for any diffeological space $X$ and subbundle $H$ of $TX$, $\PP_{H}(X)$ is the subcategory of $\PP(X)$ consisting of paths $\gamma:\RB_{+}\rightarrow M$ for which $\range(d\gamma)\subset H$.  We introduce in particular the ``leafwise path category" $\PP_{T\FF}(M)$ associated to any foliated manifold $(M,\FF)$.  Finally, we adapt from \cite{ptfunctor} the notion of a \emph{transport functor} from $\PP(X)$ to the structure groupoid of a diffeological fibre bundle $\pi_{B}:B\rightarrow X$, and introduce the \emph{holonomy groupoid} $\HH(T)$ of such a functor $T$ as the diffeological quotient of $\PP(X)$ by the ``kernel" of $T$.  We specialise this in particular to the notion of a \emph{leafwise transport functor} defined on the leafwise path category $\PP_{T\FF}(M)$ of a foliated manifold $(M,\FF)$.

In the third section, we consider a foliated manifold $(M,\FF)$ of codimension $q$.  We consider the sheaf $\DS(M,\FF)$ of so-called ``distinguished functions" \cite{holimper} on $M$, which are in effect local parametrisations of leaves.  On the set $\DS_{\g}(M,\FF)$, consisting of pairs $(x,[f]_{x})$ where $x\in M$ and where $[f]_{x}$ is the germ of a distinguished function at $x$, we give a natural diffeology under which the natural projection $\pi_{\DS_{\g}(M,\FF)}:\DS_{\g}(M,\FF)\rightarrow M$ is a diffeological fibre bundle, carrying a principal right action of the diffeological group $\g\Diff_{0}^{\loc}(\RB^{q})$ of germs at $0$ of local diffeomorphisms of $\RB^{q}$ defined in a neighbourhood of $0$.  A ``leafwise" principal partial connection $H$ is constructed for $\DS_{\g}(M,\FF)$, and we are then faced with the problem of lifting smooth paths in $M$ to paths in $\DS_{\g}(M,\FF)$ that are tangent to $H$.  We obtain the following theorem.

\begin{thm}\label{thm1}[Theorem \ref{lifting}]
	Let $(M,\FF)$ be a foliated manifold.  To each $(\gamma,d)\in\PP_{T\FF}(M)$ and each $(x,[f]_{x})\in\DS_{\g}(M,\FF)$, there exists a unique smooth map $\gamma_{[f]_{x}}:\RB_{+}\rightarrow\DS_{\g}(M,\FF)$, with $\gamma_{[f]_{x}}(0) = (x,[f]_{x})$, such that
	\begin{enumerate}
		\item $\range(d\gamma_{[f]_{x}})\subset H$, and
		\item $\pi_{\DS_{\g}(M,\FF)}\circ\gamma_{[f]_{x}} = \gamma$.
	\end{enumerate}
	The corresponding lifting map $\PP_{T\FF}(M)\times_{s,\pi_{\DS_{\g}(M,\FF)}}\DS_{\g}(M,\FF)\rightarrow \PP_{H}(\DS_{\g}(M,\FF))$ is smooth.
\end{thm}

The problem posed by Theorem \ref{thm1} is the solution of a parallel transport differential equation in a diffeological bundle.  This type of problem is considered for certain classes of infinite dimensional Lie groups in \cite{magnot0,magnot2}.  In our setting, the differential equation is solved as though the underlying geometry were a (possibly non-Euclidean) manifold - one solves the problem locally in parametrisations in whose coordinates the problem is easily soluble, and patches these parametrisations together at the end.  We see in this approach a formalisation of the old ideas of Winkelnkemper \cite{wink} and Phillips \cite{holimper} as solutions to the parallel transport problem in the diffeological bundle of germs, and we prove the following theorem.

\begin{thm}[Theorem \ref{coincide}]
	The holonomy groupoid associated to partial connection $H$ in the bundle of germs $\DS_{\g}(M,\FF)$ is isomorphic, as a diffeological groupoid, to the Winkelnkemper-Phillips holonomy groupoid.
\end{thm}

In the fourth and final section we generalise these ideas to give the first construction of the holonomy groupoid of a foliated bundle $\pi_{B}:B\rightarrow M$.  In this case, there are two ``types" of holonomy one must consider.  The first is the well-known vertical holonomy defined by lifting leafwise paths in $M$ to leafwise paths in $B$, giving rise to parallel transport maps between the fibres of $B$ arising from solutions to a differential equation of the form given in Equation \eqref{holeqn}.  This holonomy can be thought of as pertaining to parallel transport of 0-jets of sections of $B$, and we call the associated holonomy groupoid $\HH(T_{\pi_{B}^{0}})$ the ``fibre holonomy groupoid" of the foliated bundle $\pi_{B}$.  One obtains in a similar way the fibre holonomy groupoids $\HH(T_{\pi_{B}^{k}})$ of the foliated transverse jet bundles $\pi^{k}_{B}:J^{k}_{\t}(\pi_{B})\rightarrow M$.

The second is a kind of ``horizontal holonomy", which keeps track of entire transverse germs of ``distinguished sections" of $B$ that parametrise leaves in $M$ by leaves in $B$.  A natural diffeology is defined on the space $\DS_{\g}(\pi_{B})$ of germs of distinguished sections of $\pi_{B}$, for which the natural projection $\pi_{\DS_{\g}(\pi_{B})}:\DS_{\g}(\pi_{B})\rightarrow M$ is a diffeological fibre bundle.  We characterise this fibration as an associated bundle for the principal bundle $\pi_{\DS_{\g}(M,\FF)}:\DS_{\g}(M,\FF)\rightarrow M$, and exhibit a leafwise partial connection $H^{\g}_{B}$ for $\DS_{\g}(\pi_{B})$.  By similar arguments to those used in the proof of Theorem \ref{thm1}, we obtain the following theorem.
	
\begin{thm}[Theorem \ref{liftingB}]
	Let $\pi_{B}:B\rightarrow M$ be a foliated bundle.  To each $(\gamma,d)\in\PP_{T\FF}(M)$ and $(x,[\sigma]_{x})\in\DS_{\g}(\pi_{B})$, there is a unique smooth map $\gamma_{[\sigma]_{x}}:\RB_{+}\rightarrow\DS_{\g}(\pi_{B})$, with $\gamma_{[\sigma]_{x}} = (x,[\sigma]_{x})$, such that
	\begin{enumerate}
		\item $\range(d\gamma_{[\sigma]_{x}})\subset H^{\g}_{B}$, and
		\item $\pi_{\DS_{\g}(\pi_{B})}\circ\gamma_{[\sigma]_{x}} = \gamma$.
	\end{enumerate}
	The corresponding lifting map $\PP_{T\FF}(M)\times_{s,\pi_{\DS_{\g}(\pi_{B})}}\DS_{\g}(\pi_{B})\rightarrow\PP_{H^{\g}_{B}}(\DS_{\g}(\pi_{B}))$ is smooth.
\end{thm}

We refer to the holonomy groupoid $\HH(\DS_{\g}(\pi_{B}))$ arising from this theorem as the \emph{holonomy groupoid of $\pi_{B}$}.  In particular, when $B = \Fr(M/\FF)$ is the transverse frame bundle of any foliation $(M,\FF)$, then we show in Example \ref{fr} how recover the holonomy groupoid of $(M,\FF)$ as a special case of this general construction.  Finally, we prove in Theorem \ref{hierarchy} that this holonomy groupoid sits at the top of a canonical hierarchy of diffeological holonomy groupoids constructed from the foliations of the transverse jet bundles of $\pi_{B}$.  This shows in particular that while the Winkelnkemper holonomy groupoid of a foliation is the smallest Lie groupoid whose orbits are the leaves \cite{holimper}, it is by no means the smallest diffeological groupoid with this property.

\begin{thm}[Theorem \ref{hierarchy}]
	Let $\pi_{B}:B\rightarrow M$ be a foliated bundle.  Then for each $k\in\NB$, there are surjective morphisms $\Pi_{B}^{k,\g}:\HH(T_{\pi_{B}^{\g}})\rightarrow\HH(T_{\pi_{B}^{k}})$ and $\Pi_{B}^{k,k+1}:\HH(T_{\pi_{B}^{k+1}})\rightarrow\HH(T_{\pi_{B}^{k}})$ of diffeological groupoids for which $\Pi_{B}^{k,\g} = \Pi_{B}^{k,k+1}\circ\Pi_{B}^{k+1,\g}$.  Consequently we have a commuting diagram
	\begin{center}
		\begin{tikzcd}[row sep = large]
			& & & \HH(T_{\pi_{B}^{\g}}) \ar[d,"\Pi_{B}^{\infty,\g}"] & & & \\ & & & \HH(T_{\pi_{\infty}}) \ar[dl,"\Pi_{B}^{k+1,\infty}"] \ar[dr,"\Pi_{B}^{k,\infty}"'] \ar[drrr,"\Pi_{B}^{0,\infty}"'] & & &\\ & \hdots\ar[r] & \HH(T_{\pi_{B}^{k+1}}) \ar[rr,"\Pi_{B}^{k,k+1}"'] & & \HH(T_{\pi_{B}^{k}}) \ar[r] & \hdots \ar[r] & \HH(T_{\pi_{B}})
		\end{tikzcd}
	\end{center}
	of diffeological groupoids, which we refer to as the \textbf{hierarchy of holonomy groupoids} for the foliated bundle $\pi_{B}$.
\end{thm}

Let us conclude the introduction by remarking that we have \emph{deliberately} chosen to phrase our results only in the diffeological category, and do not speculate on how they may fit into the topological category.  There are several reasons for this, in addition to the relative niceness of the diffeological category and the ease of its use.  Firstly, diffeological spaces can always be equipped with a natural topology (the D-topology \cite[Section 2.8]{diffeology}) under which all plots are continuous, and it is by now becoming apparent that for the holonomy groupoids of singular foliations it is the diffeology, and not the induced D-topology, that is more fundamental \cite{VilGar1}.  Secondly, there is already some indication that the tools for which one requires topology, for instance $C^{*}$-algebras and bivariant $K$-theory, can already be made to work for the diffeological holonomy groupoids of singular foliations \cite{iakovos} by working directly with plots.  Thus using the diffeological category does not preclude building upon the extensive work done on studying foliations from the perspective of noncommutative geometry \cite{cs, ms, heitsch2, ben3, goro3, ben4, ben5, pcr, ben7, ben6, ben2, macr1, mac2}.  Finally, since de Rham cohomology works perfectly well for diffeological spaces \cite{diffeology}, the diffeological category seems to be the correct setting for better understanding the equivariant cohomology of foliated manifolds \cite{cyctrans,hopf1,goro1,diffcyc,goro2,crainic1,moscorangi, backindgeom,mosc1} in a non-\'{e}tale context \cite{mac1}.  The equivariant cohomology of foliated bundles will be presented in a future paper, where the hierarchy of holonomy groupoids will play a crucial role.

\subsection{Acknowledgements}

This research was funded by ARC grants DP170100247 and DP200100729.  I thank A. Carey and A. Rennie for their consistent interest and support.  Special thanks go to A. Carey for funding a short postdoc at The Australian National University, where most of this work was undertaken, and for many highly informative discussions over coffee, both about mathematics and everything else.

\section{Background on foliations}\label{sc1}

\subsection{Foliations and the Winkelnkemper-Phillips holonomy groupoid}

Recall that a \emph{codimension $q$ foliation} of a connected $n$-manifold $M$ is an integrable subbundle $T\FF$ of $TM$ of rank $p:=n-q$.  The maximal integral submanifolds of $T\FF$ define a family of mutually disjoint, $p$-dimensional immersed submanifolds of $M$ called \emph{leaves}, whose union is $M$.  It is convenient to describe such a foliation locally by charts of the following form.

\begin{defn}
	A \textbf{foliated chart of codimension $q$} for an $n$-manifold $M$ is a chart $(U\subset M,\psi:U\rightarrow\RB^{n})$ for $M$, whose image is equipped with the foliation arising from the trivial  foliation $\RB^{n} = \RB^{p}\times\RB^{q}$ of $\RB^{n}$.
\end{defn}

The preimages $P^{y}:=\psi^{-1}(\RB^{p}\times\{y\})$, $y\in\RB^{q}$ are called the \textbf{plaques} of $(U,\psi)$, while the preimages $S^{x}:=\psi^{-1}(\{x\}\times\RB^{q})$, $x\in\RB^{p}$, are called the \textbf{local transversals} of $(U,\psi)$.  We will often write $\psi$ as the pair of submersions $(x:U\rightarrow\RB^{p},y:U\rightarrow\RB^{q})$ defined by the compositions of $\psi$ with the projections from $\RB^{n}$ onto $\RB^{p}$ and $\RB^{q}$ respectively, so that $(U,\psi) = (U,x,y)$.

It is then a well-known fact \cite[Ch. 2]{cc1} that foliations of a manifold are in bijective correspondence with atlases of foliated charts which satisfy an appropriate glueing condition.

\begin{defn}\label{charts}
	Let $M$ be an $n$-manifold, let $0\leq q\leq n$ and define $p:=n-q$.  A \textbf{foliated atlas of codimension $q$} for $M$ is a maximal family $\{(U_{\alpha},\psi_{\alpha} = (x_{\alpha},y_{\alpha}))\}_{\alpha\in\AF}$ of foliated charts of codimension $q$ for $M$ such that for all $\alpha,\beta\in\AF$, the change of coordinates $\psi_{\alpha\beta}:=\psi_{\alpha}\circ\psi_{\beta}^{-1}:\psi_{\beta}(U_{\alpha}\cap U_{\beta})\rightarrow\psi_{\alpha}(U_{\alpha}\cap U_{\beta})\subset\RB^{p}\times\RB^{q}$ has the form
	\begin{equation}\label{compatible}
	\psi_{\alpha\beta}(x^{\beta},y^{\beta}) = (x_{\alpha\beta}(x^{\beta},y^{\beta}),y_{\alpha\beta}(y^{\beta}))
	\end{equation}
	for all $(x^{\beta},y^{\beta})\in\psi_{\beta}(U_{\alpha}\cap U_{\beta})\subset\RB^{p}\times\RB^{q}$.  The $y_{\alpha\beta}$ have the \textbf{cocycle property}
	\[
	y_{\alpha\beta} = y_{\alpha\delta}\circ y_{\delta\beta}
	\]
	on $\psi_{\beta}(U_{\alpha}\cap U_{\beta}\cap U_{\delta})$.
\end{defn}

In less notation-heavy language, a foliated atlas for a manifold is an atlas of foliated charts which are sufficiently nicely behaved that each plaque of any one member of the atlas meets at most one plaque of any other member of the atlas, and does so in an open set of both plaques.  The leaves of the associated foliation are then the classes of the equivalence relation which identifies any two points which may be connected by a path that may be covered by a chain of intersecting plaques.

Definition \ref{charts} is useful for defining the holonomy diffeomorphisms associated to leafwise paths in a foliation.  Let $\gamma$ be a path in a leaf of a foliated manifold $(M,\FF)$, with startpoint (source) $x$ and endpoint (range) $y$.  Let $\UU = \{U_{0},\dots,U_{k}\}$ be a family of foliated charts covering $\gamma$, with $x\in U_{0}$ and $y\in U_{k}$, which is a \emph{chain} in the sense that $U_{i}\cap U_{i-1}\neq\emptyset$ for all $1\leq i\leq k$.  We refer to $U_{0}$ and $U_{k}$ as the \emph{initial} and \emph{terminal} charts respectively.  Now any $x'$ sufficiently close to $x$ in the local transversal through $x$ associated to $U_{0}$ is contained in a unique plaque $P_{0}$ of $U_{0}$, and determines, for each $1\leq i\leq k$, a plaque $P_{i}$ of $U_{i}$ such that
\[
P_{i}\cap P_{i-1}\neq 0.
\]
Thanks to Definition \ref{charts}, each $P_{i}$ is uniquely determined by this constraint.  In particular $P_{k}$ intersects the local transversal through $y$ associated to $U_{k}$ in a single point $y'$, and we define
\[
H_{\gamma}^{\UU}(x'):=y'.
\]
Winkelnkemper \cite{wink} referred to $H_{\gamma}^{\UU}$ as a ``holonomy diffeomorphism", a nomenclature which, for the sake of historical consistency, we will continue to use here.  We will see in Section \ref{ss1}, however, that Winkelnkemper's holonomy diffeomorphisms differ, in general, from ``true" holonomy transport maps by an element of the structure group of a particular diffeological principal bundle.

\begin{defn}\label{holdiff}
	Let $(M,\FF)$ be a foliated manifold, $x$ and $y$ points in a leaf $L$ of $\FF$, and $\gamma$ a path connecting $x$ and $y$ that is entirely contained within $L$.  For any chain $\UU$ of foliated charts covering $\gamma$, the associated local diffeomorphism $H_{\gamma}^{\UU}$ is called the \textbf{holonomy diffeomorphism} associated to $\gamma$ and $\UU$.
\end{defn}

Winkelnkemper's goal in introducing holonomy diffeomorphisms was the construction of the \emph{graph} (or, in more modern terminology, the \emph{holonomy groupoid}) of a foliation.  Before we recall the construction of this groupoid, let us list the properties enjoyed by holonomy diffeomorphisms.

\begin{prop}\cite[p. 59]{wink}\label{props}
	Let $(M,\FF)$ be a foliated manifold.
	\begin{enumerate}
		\item Suppose $\gamma$ and $\gamma'$ are two leafwise paths with the same endpoints, and that there exist chains $\UU$ and $\UU'$ covering $\gamma$ and $\gamma'$ respectively that have the same initial and terminal charts.  If $H_{\gamma}^{\UU}$ and $H_{\gamma'}^{\UU'}$ have the same germ, then for any other pair $\VV$ and $\VV'$ of chart chains with the same initial and terminal charts covering $\gamma$ and $\gamma'$ respectively, $H_{\gamma}^{\VV}$ and $H_{\gamma'}^{\VV'}$ also have the same germ.
		\item If $\gamma$ and $\gamma'$ are homotopic leafwise paths with the same endpoints, and if $\UU$ and $\UU'$ are chains covering $\gamma$ and $\gamma'$ respectively that have the same initial and terminal charts, then $H^{\UU}_{\gamma}$ and $H^{\UU'}_{\gamma'}$ have the same germ.\qed
	\end{enumerate}
\end{prop}

As a consequence of item 1. in Proposition \ref{props} we can define an equivalence relation on the set of paths in leaves by declaring two such paths to be equivalent if and only if they have the same endpoints, and their associated holonomy diffeomorphisms, defined with respect to one (hence any) covering chain of foliated charts, have the same germ.  We denote by $[\gamma]$ the corresponding equivalence class of such a path $\gamma$.  Item 2. in Proposition \ref{props} guarantees that we can represent any such class by a smooth path which moreover has \emph{sitting instants} in the following sense \cite[Definition 2.1]{ptfunctor}.

\begin{defn}\label{sitting}
	Let $M$ be a manifold.  A smooth path $\gamma:[0,d]\rightarrow M$ is said to have \textbf{sitting instants} if $\gamma$ is constant on a neighbourhood of $0$ and of $d$.
\end{defn}

Sitting instants have not been employed in the construction of the holonomy groupoid classically.  Our purpose in introducing sitting instants is that it makes the concatenation of any two smooth paths again a smooth path.

Let us now briefly recall Winkelnkemper's construction of the holonomy groupoid $\HH(M,\FF)$ of a foliated manifold $(M,\FF)$ of codimension $q$.  Although the alternative construction due to J. Phillips \cite{holimper} is in some ways more systematic, and in fact served as inspiration for the theory we develop in this paper, elements of Winkelnkemper's construction are necessary in proving Theorem \ref{coincide} on the isomorphism between our diffeological constructions and Winkelnkemper's more topological ones.

As a set, $\HH(M,\FF)$ is the set of triples $(y,x,[\gamma])$, where $[\gamma]$ is the equivalence class of a path $\gamma$ contained in a leaf of $\FF$, and where $x$ and $y$ are its startpoint and endpoint respectively. The set $\HH(M,\FF)$ is clearly a groupoid, with range and source defined by sending $(y,x,[\gamma])\in\HH(M,\FF)$ to $y$ and $x$ respectively, with multiplication defined by
\[
(z,y,[\gamma])\cdot(y,x,[\eta]):=(z,x,[\gamma\eta]),
\]
where $\gamma\eta$ denotes the concatenation of the paths $\gamma$ and $\eta$, and with inversion defined by
\[
(y,x,[\gamma])\mapsto (x,y,[\gamma^{-1}]),
\]
where $\gamma^{-1}$ denotes the path $\gamma$ with reversed orientation.

A differential topology is defined on $\HH(M,\FF)$ as follows.  Fix $(y,x,[\gamma])\in\HH(M,\FF)$, and choose a representative $\gamma:[0,d]\rightarrow M$ of $[\gamma]$ and a chain $\UU = \{U_{0},\dots,U_{k}\}$ of foliated charts covering $\gamma$.  For any $t\in[0,d]$, let $\gamma^{t}:[0,d]\rightarrow M$ be the path defined by
\begin{equation}\label{gammat}
\gamma^{t}(s):=\gamma\bigg(\frac{ts}{d}\bigg),\hspace{7mm}s\in[0,d],
\end{equation}
and let $\UU_{t}$ be a subchain of $\UU$ that covers $\gamma^{t}$.  For $z$ contained in the domain of $H^{\UU}_{\gamma}$, we define a path $\gamma_{z}$ neighbouring $\gamma$ by
\begin{equation}\label{gammaz}
\gamma_{z}(t):=H^{\UU_{t}}_{\gamma^{t}}(z),\hspace{7mm}t\in[0,d],
\end{equation}
and we let $P_{z}$ be the plaque in $U_{0}$ through $z$ and let $P_{\gamma\cdot z}$ be the plaque in $U_{k}$ through $H_{\gamma}^{\UU}(z)$. We define a neighbourhood
\[
V_{(y,x,\gamma,\UU)}:=\{(b,a,[\eta_{z}])\in\HH(M,\FF):b\in P_{\gamma\cdot z},\,a\in P_{z}\}
\]
of $(y,x,[\gamma])$ in $\HH(M,\FF)$, where $\eta_{z}$ is any path of the form $s_{b}\gamma_{z} s_{a}$, where $s_{b}$ is any path contained entirely in the plaque $P_{\gamma\cdot z}$ joining $H_{\gamma}^{\UU}(z)$ to $b$, and where $s_{a}$ is any path contained entirely in the plaque $P_{z}$ joining $z$ to $a$.  Coordinates are defined on $V_{(y,x,\gamma,\UU)}$ which send $(b,a,[\eta]_{z})$ to $(x_{k}(b),x_{0}(a),y_{0}(a))$.

\begin{defn}\label{winkgpd}
	Let $(M,\FF)$ be a foliated manifold.  With the coordinate neighbourhoods defined above, $\HH(M,\FF)$ is a locally Hausdorff Lie groupoid which we will call the \textbf{Winkelnkemper-Phillips holonomy groupoid}.
\end{defn}

We include Phillips in the terminology for his contribution \cite{holimper} to the understanding of this groupoid.  We will see in Theorem \ref{hierarchy} that the Winkelnkemper-Phillips holonomy groupoid is the largest of an infinite hierarchy of holonomy groupoids defined by certain functors of diffeological categories.

\subsection{Foliated bundles and their jets}

We recall some basic definitions given by Kamber and Tondeur \cite{folbund}.  Here we let $M$ be an $n$-dimensional manifold\begin{comment}, and let $G$ be a finite-dimensional Lie group, with Lie algebra $\mathfrak{g}$\end{comment}.  If $\pi_{B}:B\rightarrow M$ is a fibre bundle, we denote by $VB:=\ker(d\pi_{B})$ the subbundle of $TB$ consisting of tangents to the fibres.\begin{comment}  If $\pi_{P}$ is a principal $G$-bundle we denote by $R:P\times G\rightarrow G$ the right action of $G$ on $P$.\end{comment}

\begin{defn}\label{parcon}
	Let $\pi_{B}:B\rightarrow M$ be a fibre bundle.  A \textbf{partial connection} on $\pi_{B}$ is a smooth subbundle $H$ of $TB$ for which $H\cap VB$ is equal to the zero section.
\end{defn}

\begin{defn}\label{folbund}
A fibre bundle $\pi_{B}:B\rightarrow M$ is said to be a \textbf{foliated bundle} if it admits a partial connection $T\FF_{B}$ which is involutive.
\begin{comment}
If $\pi_{P}$ is a principal $G$-bundle, we say that $\pi_{P}$ is a \textbf{foliated principal bundle} if it is a foliated bundle, whose subbundle $T\FF_{P}$ of $TP$ is $G$-equivariant in the sense that $T_{p\cdot g}\FF_{P} = (dR_{g})_{p}T_{p}\FF_{P}$ for all $p\in P$ and $g\in G$.
\end{comment}
\end{defn}

Note that if $\pi_{B}$ is a foliated bundle over $M$, then the involutivity of the distinguished subbundle $T\FF_{B}$ both integrates to a foliation $\FF_{B}$ of $B$, and induces a foliation $\FF_{M}$ of $M$ for which $T\FF_{M}$ is the involutive subbundle $(d\pi_{B})(T\FF_{B})$ of $TM$.  Kamber and Tondeur define what is meant by a connection on $\pi_{B}$ which is adapted to this foliated structure.

\begin{defn}
Let $\pi_{B}$ be a foliated bundle over $M$.  A connection $\omega\in\Omega^{1}(B;VB)$ is said to be \textbf{adapted} to $\FF_{B}$ if $T\FF_{B}\subset\ker(\omega)$.
\begin{comment}
Similarly, if $\pi_{P}$ is a foliated principal $G$-bundle over $M$, a principal connection $\omega\in\Omega^{1}(P;\mathfrak{g})$ is said to be \textbf{adapted} to $\FF_{P}$ if $T\FF_{P}\subset\ker(\omega)$.
\end{comment}
\end{defn}

Any fibre bundle itself captures only the pointwise behaviour (that is, the 0-jets) of any of its sections.  Thus a bundle can be thought of as the bundle of 0-jets of its sections.  In order to capture more detailed information about the behaviour of sections, one must consider higher order jet bundles.  Since these ideas have not yet been explicated for \emph{foliated} bundles, we give details where necessary.  They appear in the non-foliated context in, for instance, \cite{saunders}.

We will be concerned primarily with what we call \emph{distinguished sections} of foliated bundles.  In order to define them we will need the notion of \emph{distinguished functions} \cite{holimper} in foliated manifolds. 

\begin{defn}\label{distinguishedsec}
	Let $(M,\FF)$ be a foliated manifold of codimension $q$, with associated foliated atlas $\{(U_{\alpha},\psi_{\alpha})\}_{\alpha\in\AF}$.  A \textbf{distinguished function} on $(M,\FF)$ is a smooth, $\RB^{q}$-valued function $f$ defined on an open subset $\dom(f)$ of $M$, such that for each $x\in\dom(f)$ there is a foliated chart $(U_{\alpha},x_{\alpha},y_{\alpha})$ containing $x$ with $U_{\alpha}\subset\dom(f)$ such that $f|_{U_{\alpha}} = y_{\alpha}$.
	
	Let $\pi_{B}:B\rightarrow M$ be a foliated bundle, and let $q$ be the codimension of the induced foliation $\FF_{M}$ of $M$.  Let $x\in M$.  A \textbf{distinguished section} of $\pi_{B}$ is a section $\sigma$ of $\pi_{B}$ defined on some open subset $\dom(\sigma)$ of $M$, such that there is a distinguished function $f:\dom(\sigma)\rightarrow\RB^{q}$ and a smooth function $\sigma_{f}:\range(f)\rightarrow B$ such that $\sigma = \sigma_{f}\circ f$.
\end{defn}

We will be concerned later with capturing the germinal behaviour of such distinguished sections.  For now, we consider only their behaviour up to 1-jets.

\begin{defn}
Let $\pi_{B}$ be a foliated bundle over $M$, and denote the codimension of $\FF_{M}$ by $q$.  Let $S\subset \RB^{q}$ be an open neighbourhood of the origin.  Let $b\in B$ with $\pi_{B}(b) = x$, and let $\sigma$ and $\sigma'$ be distinguished sections of $\pi_{B}$ defined in a neighbourhood of $x$, with $\sigma(x) = b = \sigma'(x)$.  We say that $\sigma$ and $\sigma'$ have the same \textbf{transverse 1-jet} at $x$ if we have $d\sigma_{x} = d\sigma'_{x}$.
\end{defn}

The relation ``$\sigma\sim \sigma'$ if and only if $\sigma$ and $\sigma'$ have the same transverse 1-jet" defines an equivalence relation on the set of distinguished local sections $\sigma$ of $\pi_{B}$ defined in a neighbourhood of $x$ with $\sigma(x) = b$.  The set of equivalence classes of such sections is denoted $J^{1}_{\t}(\pi_{B})_{b}$ (the ``t" here stands for ``transverse"), and the equivalence class of such a section $\sigma$ is denoted $j^{1}_{x}(\sigma)$ and referred to as the \emph{transverse 1-jet at $x$ of the section $s$}.  Define $J^{1}_{\t}(\pi_{B}):=\bigsqcup_{b\in B}J^{1}_{\t}(\pi_{B})_{b}$ and let $\pi_{B}^{0,1}:J^{1}_{\t}(\pi_{B})\rightarrow B$ be the map that sends $J^{1}_{\t}(\pi_{B})_{b}$ to $b$.  

The set $J^{1}_{\t}(\pi_{B})$ may be equipped with a differentiable structure as follows.  Let $F$ be the typical fibre of $B$, and let us consider a foliated chart $U_{\alpha}$ in $B$, which we may assume without loss of generality to be contained in the domain of some local trivialisation $\tau_{\alpha}:B|_{\pi_{B}(U_{\alpha})}\rightarrow \pi_{B}(U_{\alpha})\times F$.  Hence we may regard $\pi_{B}(U_{\alpha})$ as a foliated chart in $M$, with foliated coordinates $\psi^{M}_{\alpha} = (x_{\alpha},y_{\alpha})$. Denoting by $f_{\alpha}$ coordinates on $(\proj_{F}\circ\tau_{\alpha})(U_{\alpha})$, and by $\tilde{f}_{\alpha}$ the composite $f_{\alpha}\circ\proj_{F}\circ\tau_{\alpha}$, our foliated coordinates on $U_{\alpha}$ can be written
\[
\psi^{B}_{\alpha} := (x_{\alpha},y_{\alpha},f_{\alpha}):U_{\alpha}\rightarrow \RB^{\dim(M)-q}\times\RB^{q}\times\RB^{\dim(F)}.
\]
Denote points in the image of these coordinates using superscripts $(x^{\alpha},y^{\alpha},f^{\alpha})$.  If $\sigma$ is a distinguished local section of $\pi_{B}$, defined in neighbourhood of $x\in \pi_{B}(U_{\alpha})$ and with $\sigma(x) = b\in U_{\alpha}$, then for notational simplicity denote the composite
\[
f_{\alpha}\circ\proj_{F}\circ\tau_{\alpha}\circ \sigma\circ\varphi_{\alpha}^{-1}:\RB^{\dim(M)-q}\times\RB^{q}\rightarrow\RB^{\dim(F)}
\]
by $f^{j^{1}\sigma}_{\alpha}$.  Then the 1-jet $j^{1}_{x}(\sigma)$ of $\sigma$ at $x$ is determined by the $q\times\dim(F)$ coordinate derivative matrix
\[
\frac{\partial f^{j^{1}\sigma}_{\alpha}}{\partial y^{\alpha}}(y_{\alpha}(x)),
\]
which depends only on the transverse coordinates $y_{\alpha}(x)$ because $\sigma$ is distinguished.  Defining $\frac{\partial\tilde{f}_{\alpha}}{\partial y_{\alpha}}:(\pi_{B}^{0,1})^{-1}(U_{\alpha})\rightarrow \RB^{q\dim(F)}$ by
\[
\frac{\partial\tilde{f}_{\alpha}}{\partial y_{\alpha}}(j^{1}_{x}(\sigma)):=\frac{\partial f^{j^{1}\sigma}_{\alpha}}{\partial y^{\alpha}}(y_{\alpha}(x)),
\]
we thereby obtain coordinates
\begin{equation}\label{coordsjet}
\psi^{J^{1}_{\t}(\pi_{B})}_{\alpha}:=\bigg(x_{\alpha},y_{\alpha},\tilde{f}_{\alpha},\frac{\partial\tilde{f}_{\alpha}}{\partial y_{\alpha}}\bigg)
\end{equation}
on $(\pi_{B}^{0,1})^{-1}(U_{\alpha})$.

\begin{rmk}\normalfont
	Note that if $\pi_{B}$ is a flat bundle, then the codimension of the induced foliation on $M$ is zero.  In this case, any two local sections of $B$ defined about $x\in M$ which have the same value at $x$ automatically have the same transverse 1-jet, and we simply have $J^{1}_{\t}(\pi_{B}) = B$.
\end{rmk}

The following result is routine to verify.

\begin{prop}\label{1jets}
	Let $\pi_{B}$ be a foliated bundle, let $N\FF_{M}\rightarrow M$ be the normal bundle associated to the induced foliation $\FF_{M}$ of $M$ and let $N^{*}\FF_{M}\rightarrow M$ be its dual.  Then the coordinates defined above give the set $J^{1}_{\t}(\pi_{B})$ a $\dim(B)+q\dim(F)$-dimensional manifold structure, with respect to which $\pi_{J^{1}_{\t}(\pi_{B}),B}:J^{1}_{\t}(\pi_{B})\rightarrow B$ is a foliated affine bundle modelled on the vector bundle $VB\otimes\pi_{B}^{*}(N^{*}\FF_{M})$ over $B$.\qed
\end{prop}
\begin{defn}
	If $\pi_{B}$ is a foliated bundle, then the bundle $\pi_{B}^{0,1}\rightarrow B$ is called the \textbf{transverse 1-jet bundle of $\pi_{B}$}.
\end{defn}

\begin{rmk}\label{jets}\normalfont
	One can in essentially the same fashion as above define the \emph{transverse $k$-jet bundle} of any foliated bundle $\pi_{B}$ by looking at $k$-jets of distinguished sections.  In this case, two distinguished sections $\sigma$ and $\sigma'$ defined around $x$, with $\sigma(x) = b = \sigma'(x)$ have the same $k$-jet at $x$ if in a foliated coordinates $(x_{\alpha},y_{\alpha},f_{\alpha})$ about $b$ in $B$, one has equality of the partial derivatives
	\[
	\frac{\partial^{|I|} f_{\alpha}^{j^{1}\sigma}}{\partial (y^{\alpha})^{I}}(y_{\alpha}(x)) = \frac{\partial^{|I|} f_{\alpha}^{j^{1}\sigma'}}{\partial (y^{\alpha})^{I}}(y_{\alpha}(x))
	\]
	for all multi-indices $I$ with $|I|\leq k$.  One defines $J^{k}_{\t}(\pi_{B})$ to be the disjoint union over all $x\in M$ of $k$-jet equivalence classes of distinguished sections at $x$, and obtains coordiantes
	\[
	\bigg(x_{\alpha},y_{\alpha},\tilde{f}_{\alpha},\frac{\partial\tilde{f}_{\alpha}}{\partial y_{\alpha}},\frac{\partial^{|I| = 2}\tilde{f}_{\alpha}}{\partial y_{\alpha}^{I}},\dots,\frac{\partial^{|I|=k}\tilde{f}_{\alpha}}{\partial y_{\alpha}^{I}}\bigg)
	\]
	defined in a similar fashion to the $k = 1$ case, giving $J^{k}_{\t}(\pi_{B})$ a manifold structure.  One then obtains a tower
	\[
	J^{k}_{\t}(\pi_{B})\xrightarrow{\pi^{k-1,k}_{B}}J^{k-1}_{\t}(\pi_{B})\xrightarrow{\pi_{B}^{k-2,k-1}}\cdots\xrightarrow{\pi_{B}^{2,1}}J^{1}_{\t}(\pi_{B})\xrightarrow{\pi_{B}^{0,1}}B
	\]
	of foliated affine bundles in a manner analogous to Proposition \ref{1jets}.  We denote by $\pi^{k}_{B}:J^{k}_{\t}(\pi_{B})\rightarrow M$ and by $\pi^{l,k}_{B}:J^{k}_{\t}(\pi_{B})\rightarrow J^{l}_{\t}(\pi_{B})$ the projections, and obtain thereby a projective system of manifolds.  We denote the projective limit of these manifolds in the diffeological category (see Definition \ref{diffs}) by $J^{\infty}_{\t}(\pi_{B})$, with projection $\pi^{\infty}_{B}:J^{\infty}_{\t}(\pi_{B})\rightarrow M$.  These higher order jet bundles play a key role in the hierarchy of holonomy groupoids that we will consider in Theorem \ref{hierarchy}.
\end{rmk}

\begin{comment}
\begin{rmk}\normalfont
	Note that if $\pi_{P}$ is a foliated principal $G$-bundle, then for each $k\geq 1$, the transverse $k$-jet bundle $\pi_{J^{k}_{\t}P}:J^{k}_{\t}P\rightarrow M$ carries a natural right action of $G$ defined by
	\[
	(j^{k}_{x}s,g)\mapsto j^{k}_{x}(s\cdot g)
	\]
	for all $j^{k}_{x}s\in J^{k}_{\t}P$ and $g\in G$.  This fact will not play a large role in the theory we develop in this paper.
\end{rmk}
\end{comment}

\section{Diffeology}

\subsection{Basic definitions}

Diffeology provides a convenient and easy differential structure for the space of smooth paths in leaves of a foliated manifold.  We recall the relevant definitions here, following \cite{diffeology}.

\begin{defn}
Let $X$ be a set.  A \textbf{parametrisation in $X$} is any function $\varphi:U\rightarrow X$ from an open subset $U$ of Euclidean space to $X$.  A \textbf{diffeology} on $X$ is a set $\DD$ of parametrisations, for which:
\begin{enumerate}
\item every constant parametrisation in $X$ is contained in $\DD$,
\item whenever $\varphi:U\rightarrow X$ is a parametrisation, and for every $u\in U$ there exists an open neighbourhood $V$ of $u$ in $U$ such that $\varphi |_{V}$ belongs to $\DD$, one has $\varphi\in \DD$, and
\item for every element $\varphi:U\rightarrow X$ of $\DD$, and for every open subset $V$ of Euclidean space and for every smooth function $f:V\rightarrow U$, one has $\varphi\circ f\in\DD$. 
\end{enumerate}
We call the tuple $(X,\DD)$ a \textbf{diffeological space}, although will usually omit the $\DD$ from notion.  We refer to the elements of $\DD$ as the \textbf{plots of $X$}.
\end{defn}

\begin{ex}\normalfont
Any smooth manifold $M$ is naturally a diffeological space, whose plots are all those smooth (in the usual sense) parametrisations $\varphi:U\rightarrow M$, where $U$ is an open subset of any Euclidean space.  Smooth manifolds are a full faithful subcategory of the category of diffeological spaces, which can be characterised as all those diffeological spaces that are locally diffeomorphic to Euclidean space.
\end{ex}

Let us give some immediate definitions of natural diffeologies associated to products, subspaes and limits.

\begin{defn}\label{diffs}
Let $X$ be a diffeological space.
\begin{enumerate}
\item If $Y$ is any other diffeological space, the product $X\times Y$ carries a natural diffeology, called the \textbf{product diffeology}, whose plots are all those parametrisations $\varphi:U\rightarrow X\times Y$ such that $\proj_{X}\circ\varphi$ and $\proj_{Y}\circ\varphi$ are plots of $X$ and $Y$ respectively.
\item If $Z$ is any subset of $X$, with inclusion $\iota:Z\hookrightarrow X$, then $Z$ inherits a natural diffeology, called the \textbf{subspace diffeology}, whose plots are all those parametrisations $\varphi:U\rightarrow Z$ for which the composite $\iota\circ\varphi:U\rightarrow X$ is a plot of $X$.
\item If $(X_{i})_{i\in I}$ is a family of diffeological spaces, and $(\phi_{ij}:X_{j}\rightarrow X_{i})_{i\leq j}$ is projective family of smooth maps, so that $\phi_{ii} = \id$ and $\phi_{ij} = \phi_{ik}\circ\phi_{kj}$, their \textbf{diffeological projective limit} is the set
\[
X:=\lim_{i}X_{i} = \{\vec{x}\in\Pi_{i\in I}X_{i}:\phi_{ij}(x_{j}) = x_{i}\text{ for all $i\in I$}\},
\]
equipped with the coarsest diffeology so that each projection $X\rightarrow X_{i}$ is smooth.
\begin{comment}
\item If $X_{\alpha}$, $\alpha\in\AF$, is any family of diffeological spaces, with $\AF$ some indexing set, then the sum
\[
X:=\bigsqcup_{\alpha\in\AF}X_{\alpha}
\]
can be equipped with a natural diffeology, called the \textbf{sum diffeology}, whose plots are those parametrisations $\rho:U\rightarrow X$, written $\rho(u) = (\alpha(u),\tilde{\rho}(u))$, such that for each $u\in U$ there exists an index $\alpha$ and a neighbourhood $V$ of $u$ in $U$ such that $\alpha(v) = \alpha$ for all $v\in V$ and $\tilde{\rho}:V\rightarrow X_{\alpha}$ is a plot of $X_{\alpha}$.
\end{comment}
\end{enumerate}
\end{defn}

Smooth functions between diffeological spaces are defined naturally as those functions whose composite with any plot is again a plot.

\begin{defn}\label{smoothfunction}
Let $X$ and $Y$ be diffeological spaces.  A function $f:X\rightarrow Y$ is said to be \textbf{smooth} if, whenever $\varphi:U\rightarrow X$ is a plot of $X$, the composition $f\circ\varphi:U\rightarrow Y$ is a plot of $Y$.
\end{defn}

Using Definition \ref{smoothfunction}, we can define a natural diffeology on the set of smooth functions between any two diffeological spaces.

\begin{defn}\label{functionaldiff}
Let $X$ and $Y$ be diffeological spaces.  The set $C^{\infty}(X, Y)$ of all smooth maps from $X$ to $Y$ can be equipped with the \textbf{functional diffeology}, whose plots are by definition all those parametrisations $\varphi:U\rightarrow C^{\infty}(X, Y)$ for which the map
\[
\tilde{\varphi}:(u, x)\mapsto\varphi(u)(x)
\]
is a smooth function $U\times X\rightarrow Y$ with respect to the product diffeology on $U\times X$.
\end{defn}

Functional diffeologies are applicable more generally to spaces of \emph{locally defined} functions also.  As we have not been able to find a reference for this fact in the literature we give the details here.  Let $M$ and $N$ be smooth manifolds.  By a \emph{locally defined smooth function} we mean a smooth function $f:\dom(f)\rightarrow N$, where $\dom(f)$ is an open subset of $M$, regarded with its inherited differentiable structure.  We denote by $C^{\infty}_{\loc}(M,N)$ the set of all locally defined smooth functions.  We declare a parametrisation $\rho:U\rightarrow C^{\infty}_{\loc}(M,N)$ to be a plot if for each $u\in U$, there exists an open neighbourhood $V$ of $u$ in $U$ and an open subset $W$ of $\dom(\rho(u))$ in $M$ such that $W\subset\dom(\rho(v))$ for all $v\in V$, and such that the map $V\times W\ni(v,x)\mapsto\rho(v)(x)\in N$ is smooth.  The arguments of the proof of \cite[1.63]{diffeology} show that the set of such plots is indeed a diffeology for $C^{\infty}_{\loc}(M,N)$.
\begin{comment} and notice that it is in bijective correspondence with the disjoint union
\[
C^{\infty}_{\sqcup}(M,N):=\bigsqcup_{U\subset M}C^{\infty}(U,N),
\]
where $U$ runs over all open subsets of $M$.  Each $C^{\infty}(U,N)$ may be equipped with the functional diffeology of Definition \ref{functionaldiff}, and the sum $C^{\infty}_{\sqcup}(M,N)$ can then be equipped with the sum diffeology of Definition \ref{diffs}.
\end{comment}

\begin{defn}\label{functionaldiff2}
	Let $M$ and $N$ be smooth manifolds.  The diffeology on $C^{\infty}_{\loc}(M,N)$ defined immediately above is called the \textbf{functional diffeology} on $C^{\infty}_{\loc}(M,N)$.
\end{defn}

The notion of a smooth function between diffeological spaces also enables the definition of pullback and pushforward diffeologies.

\begin{defn}\label{subind}
	Let $X$ and $Y$ be sets, and let $f:X\rightarrow Y$ be a function.
	\begin{enumerate}
		\item If $Y$ has a diffeology, then $X$ inherits a natural diffeology called the \textbf{pullback diffeology} from $f$, whose plots are precisely those parametrisations $\varphi:U\rightarrow X$ for which $f\circ\varphi:U\rightarrow Y$ is a plot of $Y$.  If $f$ is injective and $X$ is equipped with the pullback diffeology, $f$ is called an \textbf{induction}.
		\item If $X$ has a diffeology, then $Y$ inherits a natural diffeology called the \textbf{pushforward diffeology} from $f$, in which a parametrisation $\varphi:U\rightarrow Y$ is declared to be a plot if and only if for each $u\in U$, there exists an open neighbourhood $V$ of $u$ in $U$ such that either $\varphi|_{V}$ is constant, or there exists a plot $\tilde{\varphi}:V\rightarrow X$ such that $\varphi|_{V} = f\circ\tilde{\varphi}$.  If $f$ is surjective, and $Y$ is equipped with the pushforward diffeology, then $f$ is called a \textbf{subduction}.
	\end{enumerate}
\end{defn}
\begin{comment}
Having defined the pullback and pushforward diffeologies, let us now give two particularly important specialisations of smooth functions between diffeological spaces.

\begin{defn}
	Let $X$ and $Y$ be diffeological spaces, and let $f:X\rightarrow Y$ be a smooth map.
	\begin{enumerate}
		\item The function $f$ is said to be a \textbf{subduction} if $f$ is surjective, and if the diffeology on $Y$ is equal to the pushforward diffeology from $f$.
		\item The function $f$ is said to be an \textbf{induction} if $f$ is injective, and if the diffeology on $X$ is equal to the pullback diffeology from $f$.
	\end{enumerate}
\end{defn}
\end{comment}

Definition \ref{smoothfunction} also allows us to introduce a notion of ``diffeological category", which will play a central role in this paper.  Categories $C$ considered in this paper have objects and morphisms denoted $\obj(C)$ and $\mor(C)$ respectively, and we will always identify $\obj(C)$ with the subset of $\mor(C)$ obtained by identifying any object with its identity morphism.  We will moreover frequently identify $C$ with $\mor(C)$.

\begin{defn}
A \textbf{diffeological category} is a small category $C$ together with a diffeology $\DD$ on $\mor(C)$ for which the source $s:\mor(C)\rightarrow\obj(C)$ and range $r:\mor(C)\rightarrow\obj(C)$ are smooth maps with respect to the subspace diffeology on $\obj(C)$, and for which the composition $\mor(C)\times_{r,s}\mor(C)\rightarrow\mor(C)$ is smooth with respect to the subspace diffeology of the product diffeology on $\mor(C)\times_{r,s}\mor(C)\subset\mor(C)\times\mor(C)$.  A diffeological category $C$ is said to be a \textbf{diffeological groupoid} if $C$ is a groupoid, and the inversion is smooth.
\end{defn}

\subsection{Tangent bundles}

In giving our new definition of the holonomy groupoid of a foliation, we will need the notions of tangent spaces and tangent bundles for diffeological spaces.  There is currently no canonical definition for these objects which is suitable for all purposes.  The notion we will use is that of \emph{internal tangent spaces and bundles} defined essentially using smooth curves, which is due originally to G. Hector \cite{hector}.  The exposition we follow here is largely derived from \cite{hector, hector2, cw}.

We begin by defining what we mean by a ``diffeological vector pseudo-bundle" over a diffeological space, which is referred to in \cite{cw} as a ``vector space over a diffeological space".

\begin{defn}\label{diffvb}
	Let $X$ be a diffeological space.  A \textbf{diffeological vector pseudo-bundle over $X$} is a diffeological space $V$, a smooth map $\pi_{V}:V\rightarrow X$, and a vector space structure on each fibre $\pi_{V}^{-1}(\{x\})$, $x\in X$, such that the addition map $V\times_{X}V\rightarrow V$, the scalar multiplication map $\RB\times V\rightarrow V$ and the zero section $X\rightarrow V$ are all smooth.
\end{defn}

Note that we use the terminology ``pseudo-bundle" instead of ``bundle" since such objects need not be locally trivial - indeed, such objects need not even have a typical fibre \cite{cw}.  Let $X$ be a diffeological space and let $x\in X$.  We denote by $\GG(X,x)$ the \emph{category of plots centred at $x$}, whose objects are plots $\rho:U\rightarrow X$, where $U$ is a connected open neighbourhood of $0$ and where $\rho(0) = x$, and whose morphisms $[\phi]_{0}\in\mor_{\GG(X,x)}(\rho:U\rightarrow X,\rho':U'\rightarrow X)$ are germs of smooth functions $\phi:U\rightarrow U'$ at zero such that $[\rho]_{0} = [\rho'\circ\phi]_{0}$.  There is then a functor $F:\GG(X,x)\rightarrow\Vect$ into the category of vector spaces and smooth maps, which sends any object $\rho:U\rightarrow X$ to the vector space $T_{0}U$, and sends any morphism $[\phi]_{0}$ in  $\mor_{\GG(X,x)}(\rho:U\rightarrow X,\rho':U'\rightarrow X)$ to the linear map $d\phi_{0}:T_{0}U\rightarrow T_{0}U'$.

\begin{defn}
	Let $X$ be a diffeological space and let $x\in X$.  The \textbf{internal tangent space} $T_{x}X$ at $x$ is the colimit of the functor $\GG(X,x)\rightarrow\Vect$.
\end{defn}

Thus, as a set, the internal tangent space $T_{x}X$ of a pointed diffeological space $(X,x)$ may be described as follows.  Let
\[
\TT_{x}X:=\bigoplus_{\rho\in\obj(\GG(X,x))}\TT_{\rho}
\]
be the direct sum over the tangent spaces $\TT_{\rho}:=T_{0}U$ of plots $\rho:U\rightarrow X$ centered at $X$, and for each such $\rho$ let $j_{\rho}:T_{\rho}\rightarrow \TT_{x}X$ be the inclusion.   Then $T_{x}X$ is the quotient of $\TT_{x}X$ by the subspace generated by all vectors of the form
\begin{equation}\label{equiv}
j_{\rho}(\xi) - j_{\rho'}\big((d\phi)_{0}(\xi)\big),\hspace{7mm}\xi\in \TT_{\rho},\,[\phi]_{0}\in\mor_{\GG(X,x)}(\rho,\rho').
\end{equation}
Given a plot $\rho:U\rightarrow X$ centered at $x$ and a tangent vector $\xi\in \TT_{\rho}$, we denote by $\rho_{*}(\xi)$ the class of the element $j_{\rho}(\xi)$ in $T_{x}X$.  Clearly then $\rho_{*}:\TT_{\rho}\rightarrow T_{x}X$ is a linear map, and if $\rho$ is a constant plot then $\rho_{*}(\xi) = 0$ for all $\xi\in T_{\rho}$.

This categorical definition yields the familiar notion of tangent vectors as arising from smooth curves.

\begin{prop}\cite[Proposition 3.3]{cw}
	Let $X$ be a diffeological space and let $x\in X$.  Then every element of $T_{x}X$ is a finite linear combination of vectors $\rho_{*}(\partial_{t})$, where $\rho:U\rightarrow X$ is a plot centred at $x$, and where $\partial_{t}\in T_{0}U$ is the standard basis vector.\qed
\end{prop}

One has a natural tangent map associated to any smooth function of diffeological spaces, defined as follows.

\begin{defn}\label{pushforward}
	Let $f:X\rightarrow Y$ be a smooth map of diffeological spaces, and let $x\in X$.  The \textbf{tangent map of $f$ at $x$} is the function $df_{x}:T_{x}X\rightarrow T_{f(x)}Y$ defined by
	\[
	df_{x}\big(\rho_{*}(\partial_{t})\big):=(f\circ\rho)_{*}(\partial_{t})
	\]
	for any plot $\rho$ of $X$ centered at $x$, and extended by linearity.
\end{defn}

The construction of the tangent bundle of a diffeological space now follows in an analogous way as for manifolds.  Namely, one defines the tangent bundle to be the disjoint union of the tangent spaces and declares the tangent maps to the plots to be plots of the tangent bundle.

\begin{defn}
	Let $X$ be a diffeological space, and consider the set
	\[
	TX:=\bigsqcup_{x\in X}T_{x}X.
	\]
	For any plot $\rho:U\rightarrow X$ of $X$ and $u\in U$, let $\tau_{u}:v\mapsto v+u$ denote translation by the vector $u$ so as to obtain from $U$ a neighbourhood $\tau_{u}^{-1}U$ of $0$ and a plot $\rho\circ\tau_{u}:\tau_{u}^{-1}(U)\rightarrow X$ centred at $x$.  Define $d\rho:TU\rightarrow TX$ by the formula
	\[
	d\rho(u,\xi):=\big(\rho(u),(\rho\circ\tau_{u})_{*}(\xi)\big),\hspace{7mm}(u,\xi)\in TU = U\times\RB^{n}.
	\]
	The \textbf{Hector diffeology} on $TX$ is the diffeology generated by the $d\rho:TU\rightarrow TX$ defined for all plots $\rho:U\rightarrow X$.  The \textbf{dvs diffeology} on $TX$ is the smallest diffeology containing the Hector diffeology for which $\pi_{TX}:TX\rightarrow X$ is a diffeological vector pseudo-bundle.
\end{defn}

The reason for using the dvs diffeology on the tangent bundle $TX$ of $X$ is that the fibrewise addition and scalar multiplication operations \emph{need not} be smooth for the Hector diffeology \cite{cw}.  This is essentially due to the fact that plots of the form $d\rho:TU\rightarrow TX$ need not cover $TX$.  Plots $\varphi:U\rightarrow TX$ in the dvs diffeology are characterised by the property that for each $u\in U$, there is a neighbourhood $V$ of $u$ in $U$ and a plot $x:V\rightarrow X$, a finite family of plots $\rho_{i}:V\rightarrow TX$ from the Hector diffeology such that $\pi_{TX}\circ\rho_{i} = x$, and a finite family of smooth maps $r_{i}:V\rightarrow\RB$ such that
\[
\varphi(v) = \sum_{i=1}^{n}r_{i}(v)\rho_{i}(v)
\]
for all $v\in V$.

Finally, we have the obvious extension of Definition \ref{pushforward} to full tangent bundles.

\begin{defn}\label{differential}
	Let $f:X\rightarrow Y$ be a smooth map of diffeological spaces.  Then the map $df:TX\rightarrow TY$ defined by
	\[
	df(x,\rho_{*}(\partial_{t})):=(f(x),(f\circ\rho)_{*}(\partial_{t}))
	\]
	for any plot $\rho:U\rightarrow X$ centred at $x$, and extended by linearity in each fibre, is called the \textbf{pushforward of $f$}.
\end{defn}

By \cite[Proposition 4.8]{cw}, if $f:X\rightarrow Y$ is a smooth map of diffeological spaces, then $df:TX\rightarrow TY$ is a smooth map of their tangent bundles. 

\subsection{Diffeological groups, principal bundles, and partial connections}

A key structure in the theory developed in this paper is that of a diffeological fibre bundle (referred to in \cite{diffeology} as a ``diffeological fibration").  In order to define it, we follow \cite[Chapter 8]{diffeology} in first recalling the structure groupoid of a diffeological surjection, which is instrumental in the definition of our holonomy groupoids.

Let $\pi_{B}:B\rightarrow X$ be a smooth surjection of diffeological spaces.  For each pair $x,y\in X$, denote by $\Diff(B_{x},B_{y})$ the set of (diffeological) diffeomorphisms from the fibre $B_{x}$ onto the fibre $B_{y}$.  Denote by $\Aut(\pi_{B})$ the category for which
\[
\obj(\Aut(\pi_{B})) = X
\]
and
\[
\mor_{\Aut(\pi_{B})}(x,y) = \Diff(B_{x},B_{y}),\hspace{7mm}(x,y)\in X\times X.
\]
Composition in this category is simply composition of diffeomorphisms, and since diffeomorphisms are invertible $\Aut(\pi_{B})$ is of course a groupoid.  As is usual, we will identify the groupoid $\Aut(\pi_{B})$ with its set of morphisms, where each $x$ in the unit space $X$ can be identified with the identity morphism $\id_{B_{x}}\in\Diff(B_{x},B_{x})$.  Let us denote the range and source of $\Aut(\pi_{B})$ by $r$ and $s$ respectively.

The category $\Aut(\pi_{B})$ may be equipped with a diffeology by declaring a parametrisation $\varphi:U\rightarrow\Aut(\pi)$ to be a plot if and only if the maps
\begin{equation}\label{f1}
\varphi_{s}:U\times_{s\circ\varphi,\pi_{B}}B\ni(u,b)\mapsto\varphi(u)(b)\in B,
\end{equation}
\begin{equation}\label{f2}
\varphi_{r}:U\times_{r\circ\varphi,\pi_{B}}B\ni(u,b)\mapsto\varphi(u)^{-1}(b)\in B,
\end{equation}
and
\begin{equation}\label{f3}
(r,s)\circ\varphi:U\rightarrow X\times X
\end{equation}
are smooth.  This is the \emph{functional diffeology} on $\Aut(\pi_{B})$ - it is the smallest diffeology under which $\Aut(\pi_{B})$ is a diffeological groupoid and for which the evaluation map $\ev:\Aut(\pi_{B})\times_{s,\pi_{B}}B\ni(f,b)\mapsto f(b)\in B$ is smooth.

\begin{defn}\label{structuregroupoid}
	Let $\pi_{B}:B\rightarrow X$ be a smooth surjection of diffeological spaces.  The groupoid $\Aut(\pi_{B})$, equipped with the functional diffeology, is called the \textbf{structure groupoid of $\pi_{B}$}.
\end{defn}

Diffeological fibre bundles are those diffeological surjections whose structure groupoids are sufficiently well-behaved.

\begin{defn}
	A smooth surjection $\pi_{B}:B\rightarrow X$ is called a \textbf{diffeological fibre bundle} if the characteristic map $(r,s):\Aut(\pi_{B})\rightarrow X\times X$ of its structure groupoid is a subduction.
\end{defn}

That the characteristic map $(r,s):\Aut(\pi_{B})\rightarrow X\times X$ of the structure groupoid of a diffeological surjection is a subduction means that all fibres of $\pi_{B}$ are diffeomorphic.  Thus the notion of ``typical fibre" makes sense for a diffeological fibre bundle, but need not for a general diffeological surjection.

An important subclass of diffeological fibre bundles are diffeological principal bundles. Such objects have \emph{diffeological groups} as their fibres.

\begin{defn}
	A \textbf{diffeological group} is a group $G$ that is equipped with a diffeology for which the multiplication and inversion maps are smooth.
\end{defn}

Diffeological principal bundles associated to diffeological groups are now defined in terms of inductions (see Definition \ref{subind}).  We follow \cite[Sec. 8.11]{diffeology} for our definition, noting that the slightly more restrictive class of \emph{$D$-numerable diffeological principal bundles} has been studied and classified in \cite{magnot3,cw2}.

\begin{defn}\label{principal}
	Let $G$ be a diffeological group.  A smooth right action $R$ of $G$ on a diffeological space $P$ is said to be \textbf{principal} if the map
	\[
	P\times G\ni(p,g)\mapsto(p,R_{g}(p))\in P\times P
	\]
	is an induction.  Given such an action, equip $P/G$ with the pushfoward diffeology from the quotient $q:P\rightarrow P/G$, and declare a diffeological surjection $\pi_{P}:P\rightarrow X$ to be a \textbf{diffeological principal $G$-bundle} if there exists a diffeomorphism $f:X\rightarrow P/G$ such that $f\circ\pi_{P} = q$.
\end{defn}

If $\pi_{P}:P\rightarrow X$ is a diffeological principal $G$-bundle, then the action of $G$ on $P$ is automatically free.  Moreover the surjection $\pi_{P}:P\rightarrow X$ is a diffeological fibre bundle with typical fibre $G$ \cite[p. 243]{diffeology}.

For the construction of holonomy groupoids, we will be concerned with \emph{partial connections} of such principal bundles, which are defined in an analogous manner to Definition \ref{parcon}.  This definition does not currently appear in the literature.

\begin{defn}\label{conndiff}
	Let $\pi_{B}:B\rightarrow X$ be a diffeological fibre bundle.  A \textbf{partial connection} for $\pi_{B}$ is a subbundle $H$ of $TB$ such that, for each $b\in B$, the differential $d\pi_{B}$ restricts to a linear injection $H_{b}\rightarrow T_{\pi_{B}(b)}X$.
	
	If in particular $G$ is a diffeological group and $\pi_{P}:P\rightarrow X$ is a diffeological principal $G$-bundle, with principal action $R:P\times G\rightarrow P$, a partial connection $H$ on $\pi_{P}$ is called a \textbf{principal partial connection} if it is equivariant under the action of $G$.  That is, denoting by $dR_{g}:TP\rightarrow TP$ the differential of the action of $g\in G$ on $P$, $H$ is said to be principal if $dR_{g}(H_{p})= H_{p\cdot g}$.
\end{defn}

``Full" connections of course could be defined in a similar way, but will not be required for our constructions in this paper.

\subsection{The diffeological Moore path category}

For the entirety of this subsection we denote by $X$ a diffeological space, and by $\RB_{+} = [0,\infty)$ the set of non-negative real numbers, equipped with the subspace diffeology from $\RB$ (that is, we consider $\RB_{+}$ as a manifold with boundary).  We extend the Moore path category from algebraic topology \cite[Ch. III, Sec. 2]{whitehead} to the setting of diffeological spaces.

\begin{defn}
	The \textbf{Moore path category of $X$} is the small category $\PP(X)$ with
	\begin{enumerate}
		\item object set $\obj(\PP(X)) = X$,
		\item morphism set $\mor(\PP(X))$ consisting of pairs $(\gamma,d)$, where $d\in\RB_{+}$ is called the \textbf{duration}, and $\gamma:[0,\infty)\rightarrow X$ is a smooth map \textbf{with sitting instants}, in the sense that it is constant in a neighbourhood of $0$ and of $[d,\infty)$.  The source of $(\gamma, d)$ is $\gamma(0)$ and its range is $\gamma(d)$.  Composition of morphisms is defined by $(\gamma_{1},d_{1})\circ(\gamma_{2},d_{2}):=(\gamma_{1}\gamma_{2}, d_{1}+d_{2})$, where $\gamma_{1}\gamma_{2}$ is the path
		\[
		\gamma_{1}\gamma_{2}(t) = \begin{cases}
		\gamma_{2}(t),\,&\text{for $t\leq d_{2}$}\\
		\gamma_{1}(t-r_{2}),\,&\text{for $d_{2}\leq t<\infty$}
		\end{cases}.
		\]
		To any $x\in X = \obj(\PP(X))$, the corresponding identity morphism is $(\gamma_{x},0)\in\mor(\PP(X))$, where $\gamma_{x}$ is the constant map sending $\RB_{+}$ to $x$.
	\end{enumerate}
\end{defn}

Note our departure from the usual Moore path category in requiring that our paths have sitting instants (see Definition \ref{sitting}) - this is so that concatenation of smooth paths remains smooth.  The inclusion of durations in the definition of the Moore path category means that the concatenation of paths therein is associative on-the-nose, in contrast with the space one would obtain by considering only paths defined on $[0,1]$, for instance.  Since we are interested in doing calculus with the leafwise path category, we must equip it with a differential structure.  Diffeology provides the easiest way of doing this.

\begin{prop}
	Equip $C^{\infty}(\RB_{+},X)\times\RB_{+}$ with the product diffeology arising from the functional diffeology on $C^{\infty}(\RB_{+},X)$ and the standard diffeology on $\RB_{+}$.  Then $\PP(X)\subset C^{\infty}(\RB_{+},X)\times\RB_{+}$ inherits the subspace diffeology to become a diffeological category.
\end{prop}

\begin{proof}
	Let us first show that the source $s:\mor(\PP(X))\rightarrow\obj(\PP(X))$ is smooth.  The proof for the range is similar.  Suppressing the inclusions of $\obj(\PP(X))$ into $\mor(\PP(X))$ and of $\PP(X)$ into $C^{\infty}(\RB_{+};X)\times\RB_{+}$ for simplicity, we must show that for any plot $\varphi:U\rightarrow\PP(X)$, each of the composites
	\[
	\proj_{C^{\infty}(\RB_{+};X)}\circ s\circ\varphi:U\rightarrow C^{\infty}(\RB_{+};X),\hspace{7mm}\proj_{\RB_{+}}\circ s\circ\varphi:U\rightarrow\RB_{+}
	\]
	are plots.  We easily see that $\proj_{\RB_{+}}\circ s\circ\varphi$ is the constant map $U\ni u\mapsto 0\in\RB_{+}$, so is a plot by definition.  On the other hand, for $u\in U$ we compute
	\[
	(\proj_{C^{\infty}(\RB_{+};X)}\circ s\circ\varphi)(u) = \varphi(u)(0) = (\ev_{0}\circ\proj_{C^{\infty}(\RB_{+};X)}\circ\varphi)(u),
	\]
	where $\ev_{0}:C^{\infty}(\RB_{+};X)\rightarrow X$ is the evaluation map $f\mapsto f(x)$.  Since the evaluation maps are always smooth in the functional diffeology \cite[Sec. 1.57]{diffeology}, we conclude that $\proj_{C^{\infty}(\RB_{+};X)}\circ s\circ\varphi$ is indeed a plot, and therefore $s:\mor(\PP(X))\rightarrow\obj(\PP(X))$ is smooth.
	
	It remains only to show that the composition of morphisms $m:\mor(\PP(X))\times_{r,s}\mor(\PP(X))\rightarrow\mor(\PP(X))$ is smooth.  Let therefore $\varphi:U\rightarrow \mor(\PP(X))\times_{r,s}\mor(\PP(X))$ be a plot, so that composites with the projections $\proj_{1}\circ\varphi:U\rightarrow\mor(\PP(X))$ and $\proj_{2}\circ\varphi:U\rightarrow\mor(\PP(X))$ onto the first and second factors respectively are plots of $\mor(\PP(X))$.  We must show that $m\circ\varphi:U\rightarrow\mor(\PP(X))$ is a plot.  For ease of notation, let us write
	\[
	d_{i}(u):=(\proj_{\RB_{+}}\circ\proj_{i}\circ\varphi)(u)
	\]
	for the durations of $(\proj_{i}\circ\varphi)(u)\in\mor(\PP(X))$, $i = 1,2$.  We then compute
	\[
	(m\circ\varphi)(u)(t) = \begin{cases}
	(\proj_{2}\circ\varphi)(u)(t),\,&\text{for $0\leq t\leq d_{2}(u)$}\\
	(\proj_{1}\circ\varphi)(u)(t - d_{2}(u)),\,&\text{for $d_{2}(u)\leq t<\infty$}.
	\end{cases}
	\]
	Since the $\proj_{i}\circ\varphi$ are plots, it follows that the function $\widetilde{m\circ\varphi}:U\times\RB_{+}\ni(u,t)\mapsto(m\circ\varphi)(u)(t)\in X$ is smooth, hence that $m\circ\varphi$ is a plot of $\mor(\PP(X))$, giving smoothness of $m$ as claimed.
\end{proof}

Of course, paths can be ``inverted" by simply reversing their orientation.  While this operation does not define a genuine inversion operation on the Moore path category, it is smooth with respect to the diffeology thereon and, as we will see, descends to give inverses in diffeological quotients of geometric interest.

\begin{prop}\label{inverses}
	For a pair $(\gamma,d)\in\PP(X)$, define a new pair $(\gamma^{-1},d)\in\PP(X)$ by
	\[
	\gamma^{-1}(t):=\begin{cases}
	\gamma(d-t)\,&\text{for $0\leq t\leq d$}\\
	\gamma(0)\,&\text{for $t\geq d$}
	\end{cases}.
	\]
	Then the map $\iota:(\gamma,d)\mapsto(\gamma^{-1},d)$ is smooth.
\end{prop}

\begin{proof}
	Let $\varphi:U\rightarrow\PP(X)$ is a plot, so that the map
	\[
	U\times[0,\infty)\ni(u,t)\mapsto\varphi(u)(t)\in X
	\]
	is smooth, and let $\tilde{d}:U\rightarrow[0,\infty)$ be the smooth function which assigns to each $u\in U$ the duration of the path $\varphi(u)$.  Then
	\[
	\iota\circ\varphi(u)(t) = \begin{cases}
	\varphi(u)(\tilde{d}(u) - t)\,&\text{for $0\leq t\leq \tilde{d}(u)$}\\
	\varphi(u)(0)\,&\text{for $\tilde{d}(u)\leq t < \infty$}
	\end{cases}
	\]
	is also smooth.  Hence $\iota:\PP(X)\rightarrow\PP(X)$ is a smooth map.
\end{proof}

Since we will be working with foliations, we will be primarily interested in paths whose tangent vectors all lie in some distinguished subbundle of the tangent bundle $TX$.

\begin{defn}\label{Hpath}
	Let $H$ be a subbundle of $TX$.  The \textbf{$H$-path category} is the diffeological subcategory $\PP_{H}(X)$ of $\PP(X)$ consisting of those $(\gamma,d)$ for which $\im(d\gamma)\subset H$ (see Definition \ref{differential}).
\end{defn}

An immediate example of an $H$-path category that will be used constantly in this article is the leafwise path category of a foliation.

\begin{defn}
	Let $(M,\FF)$ be a foliated manifold.  The \textbf{leafwise path category} is the diffeological subcategory $\PP_{T\FF}(M)$ of $\PP(M)$.
\end{defn}

The following definition is inspired by the systematic description of transport functors given in \cite{ptfunctor}.

\begin{defn}\label{transportfunctor}
	Let $X$ be a diffeological space, and $\pi_{B}:B\rightarrow X$ a diffeological fibre bundle.  A smooth functor $T:\PP(X)\rightarrow \Aut(\pi_{B})$ is said to be a \textbf{transport functor} if there exists a smooth lifting map
	\[
	L:\PP(X)\times_{s,\pi_{B}}B\rightarrow\PP(B)
	\]
	such that $T$ can be written as the composite
	\[
	T(\gamma,d)(b):=r\circ L\big((\gamma,d),b\big),\hspace{7mm}\big((\gamma,d),b\big)\in\PP(X)\times_{s,\pi_{B}}B.
	\]
	If in particular $T:\PP_{T\FF}(M)\rightarrow\Aut(\pi_{B})$ for some foliated manifold $(M,\FF)$, we refer to $T$ as a \textbf{leafwise transport functor}.
\end{defn}

Quotienting a path space by ``the kernel" of a transport functor yields a diffeological groupoid.

\begin{prop}
	Let $X$ be a diffeological space and let $\pi_{B}:B\rightarrow X$ be a diffeological fibre bundle.  Let $T:\PP(X)\rightarrow\Aut(\pi_{B})$ be a transport functor.  The quotient $\HH(T):=\PP(X)/\sim_{T}$ of $\PP(X)$ by the relation $\sim_{T}$, where $(\gamma_{1},d_{1})\sim_{T}(\gamma_{2},d_{2})$ if and only if $T(\gamma_{1},d_{1}) = T(\gamma_{2},d_{2})$, inherits from $\PP(X)$ the structure of a diffeological groupoid, whose inversion is induced by the smooth map $\iota:\PP(X)\rightarrow\PP(X)$ of Proposition \ref{inverses}.
\end{prop}

\begin{proof}
	Since $T$ is a functor it preserves the composition of morphisms in $\PP(X)$, so that the formula
	\[
	[(\gamma_{1},d_{1})][(\gamma_{2},d_{2})]:=[(\gamma_{1}\gamma_{2},d_{1}+d_{2})]
	\]
	defines a smooth, associative multiplication on $\HH(T)$.  To see that $\HH(T)$ is a groupoid then it suffices to show that the map $\iota:(\gamma,d)\mapsto(\gamma^{-1},d)$ is mapped by $T$ to the inversion in the groupoid $\Aut(\pi_{B})$.  However this follows from the factorisation $T = r\circ L$ of Definition \ref{transportfunctor} and the fact that $r\circ\iota = s$ on $\PP(X)$.
\end{proof}

Note that leafwise transport functors also induce diffeological groupoids as quotients of the smaller category $\PP_{T\FF}(M)$.  These are the groupoids that will be of primary interest in this paper.

\begin{defn}
	Let $X$ be a diffeological space, $\pi:B\rightarrow X$ a diffeological fibre bundle, and $T:\PP(X)\rightarrow\Aut(\pi_{B})$ a transport functor.  We refer to the associated diffeological groupoid $\HH(T)$ as the \textbf{holonomy groupoid associated to $T$}.
\end{defn}

We will see in the next section that we can recover the Winkelnkemper-Phillips holonomy groupoid of a foliation $(M,\FF)$ from a canonical leafwise transport functor in the sense of Definition \ref{transportfunctor}.  Our constructions can be seen as a rigorous justification for the nomenclature ``holonomy" in referring to the Winkelnkemper-Phillips holonomy groupoid.

\section{The holonomy groupoid of a foliated manifold}\label{ss1}

We can associate to any foliated manifold a canonical transport functor, whose associated holonomy groupoid is the Winkelnkemper-Phillips holonomy groupoid.  This transport functor is inspired by J. Phillips' careful work in \cite{holimper}.

\subsection{A principal bundle of germs}

For the entirety of this section, let $(M,\FF)$ be a foliated manifold of codimension $q$.  Recall (Definition \ref{distinguishedsec}) that we say that a smooth, $\RB^{q}$-valued function $f$ defined on an open subset $\dom(f)$ of $M$ is a \emph{distinguished function on $M$} if for each $x\in \dom(f)$, there exists a foliated chart $(U_{\alpha},x_{\alpha},y_{\alpha})$ about $x$ in $\dom(f)$ such that $f|_{U_{\alpha}} = y_{\alpha}$.  We denote by $\DS(M,\FF)$ the set of all distinguished functions on $M$, equipped with the functional diffeology of Definition \ref{functionaldiff2}. We construct a diffeological fibre bundle of germs of distinguished functions as follows.

Let us consider the diffeological product $M\times\DS(M,\FF)$, and the diffeological subset
\[
S:=\{(x,f)\in M\times\DS(M,\FF):x\in\dom(f),\,f(x) = 0\}.
\]
We declare two points $(x,f)$ and $(y,g)$ in $S$ to be equivalent, written $(x,f)\sim(y,g)$, if and only if $x = y$ and the germs $[f]_{x}$ and $[g]_{x}$ are equal, and denote by
\[
\DS_{\g}(M,\FF):=S/\sim
\]
the corresponding diffeological quotient.  Clearly then $\DS_{\g}(M,\FF)$ consists of pairs $(x,[f]_{x})$, where $x\in M$ and where $[f]_{x}$ is the germ at $x$ of some distinguished function $f$ defined in a neighbourhood of $x$, with $f(x) =0$. Moreover, a parametrisation $\rho:U\rightarrow \DS_{\g}(M,\FF)$ is a plot if and only if for each $u\in U$ there exists a neighbourhood $V$ of $u$ in $U$, and plots $\tilde{x}:V\rightarrow M$, $\tilde{f}:V\rightarrow\DS(M,\FF)$ such that $\tilde{x}(v)\in\dom(\tilde{f}(v))$, $\tilde{f}(v)\big(\tilde{x}(v)\big) = 0$, and such that
\[
\rho(v) = (\tilde{x}(v),[\tilde{f}(v)]_{\tilde{x}(v)})
\]
for all $v\in V$.  We denote by $\pi_{\DS_{\g}(M,\FF)}:\DS_{\g}(M,\FF)\rightarrow M$ the canonical surjective map defined by
\[
\pi_{\DS_{\g}(M,\FF)}(x,[f]_{x}):=x,\hspace{7mm}(x,[f]_{x})\in\DS_{\g}(M,\FF).
\]
The following result is then elemetary to verify.

\begin{lemma}
	Let $(M,\FF)$ be a foliated manifold.  Then the map $\pi_{\DS_{\g}(M,\FF)}:\DS_{\g}(M,\FF)\rightarrow M$ is a subduction.
\end{lemma}

\begin{proof}
	The map $\pi_{\DS_{\g}(M,\FF)}$ is clearly surjective, so we need only show that the standard diffeology on $M$ is the pushforward diffeology from $\pi_{\DS_{\g}(M,\FF)}$.  Let $\tilde{x}:U\rightarrow M$ be a plot.  For each $u\in U$, we can find a sufficiently small neighbourhood $V$ of $u$ in $U$ such that $\tilde{x}|_{V}$ takes values in some foliated chart $(U_{\alpha},x_{\alpha},y_{\alpha})$ containing $\tilde{x}(u)$.  For each $v\in V$, we define
	\[
	\tilde{y}_{\alpha}(v)\big(x\big):=y_{\alpha}(x) - y_{\alpha}(\tilde{x}(v)),
	\]
	and observe then that each $(U_{\alpha},x_{\alpha},\tilde{y}_{\alpha}(v))$ satisfies the compatibility condition given in Equation \eqref{compatible} with respect to $(U_{\alpha},x_{\alpha},y_{\alpha})$, hence is a member of the foliated atlas associated to $(M,\FF)$.  We denote by $\tilde{y}_{\alpha}:V\rightarrow \DS(M,\FF)$ the plot sending $v\in V$ to the distinguished function $\tilde{y}_{\alpha}(v)$, and then observe that the plot $\rho:V\rightarrow \DS_{\g}(M,\FF)$ defined by
	\[
	\rho(v):=(\tilde{x}(v),[\tilde{y}_{\alpha}(v)]_{\tilde{x}(v)}),\hspace{7mm}v\in V
	\]
	satisfies $\pi_{\DS_{\g}(M,\FF)}\circ\rho = \tilde{x}|_{V}$.
\end{proof}

The space $\DS_{\g}(M,\FF)$ carries additional structure.  Denote by $\Diff^{\loc}_{0}(\RB^{q})$ the space of local diffeomorphisms of $\RB^{q}$ that are defined in a neighbourhood of $0$ and that fix $0$, equipped with the functional diffeology of Definition \ref{functionaldiff2}.  Denote by $\g\Diff^{\loc}_{0}(\RB^{q})$ the diffeological quotient of $\Diff^{\loc}_{0}(\RB^{q})$ that identifies two local diffeomorphisms if and only if they have the same germ at $0$.  Thus $\g\Diff^{\loc}_{0}(\RB^{q})$ is a diffeological group under composition of germs.  Now the group $\g\Diff^{\loc}_{0}(\RB^{q})$ acts canonically on the right of $\DS_{\g}(M,\FF)$ according to the formula
\[
(x,[f]_{x})\cdot[\varphi]_{0}:=(x,[\varphi^{-1}\circ f]_{x}),\hspace{7mm}(x,[f]_{x})\in\DS_{\g}(M,\FF),\,[\varphi]_{0}\in\g\Diff^{\loc}_{0}(\RB^{q})
\]
where the composite $\varphi^{-1}\circ f$ is taken on any open subset of the domain of $f$ on which this makes sense.  Since our foliated atlases are by definition maximal (see Definition \ref{charts}), this right action is principal in the sense of Definition \ref{principal}.

\begin{prop}
	Let $(M,\FF)$ be a foliated manifold of codimension $q$.  Then $\pi_{\DS_{\g}(M,\FF)}:\DS_{\g}(M,\FF)\rightarrow M$ is a principal $\g\Diff^{\loc}_{0}(\RB^{q})$-bundle.
\end{prop}

\begin{proof}
	We need to show that the map $A:\DS_{\g}(M,\FF)\times\g\Diff^{\loc}_{0}(\RB^{q})\rightarrow\DS_{\g}(M,\FF)\times\DS_{\g}(M,\FF)$ defined by
	\[
\big((x,[f]_{x}),[\varphi]_{0}\big)\mapsto\big((x,[f]_{x}),(x,[f]_{x})\cdot[\varphi]_{0}\big),\hspace{7mm}\big((x,[f]_{x}),[\varphi]_{0}\big)\in\DS_{\g}(M,\FF)\times\g\Diff^{\loc}_{0}(\RB^{q}),
	\]
	is an induction.  Observe first that since the foliated atlas associated to $(M,\FF)$ is maximal, for any $(x,[y_{\alpha}]_{x})$ and $(x,[y_{\beta}]_{x})$ in $\DS_{\g}(M,\FF)$, deriving from transverse coordinates $y_{\alpha}$ and $y_{\beta}$ about $x$ respectively, with $y_{\alpha}(x) = 0 = y_{\beta}(x)$, the germ $[y_{\alpha\beta}]_{0}\in\g\Diff^{\loc}_{0}(\RB^{q})$ arising from the transverse coordinate change satisfies
	\[
	(x,[y_{\alpha}]_{x})\cdot[y_{\alpha\beta}]_{0} = (x,[y_{\beta\alpha}\circ y_{\alpha}]_{x}) = (x,[y_{\beta}]_{x}),
	\]
	and is uniquely determined in $\g\Diff^{\loc}_{0}(\RB^{q})$ by this property.  It follows immediately that $A$ is an injective map, and it remains only to show that the product diffeology on $\DS_{\g}(M,\FF)\times\g\Diff^{\loc}_{0}(\RB^{q})$ coincides with the pullback diffeology from $A$.  Let $\rho:U\rightarrow\DS_{\g}(M,\FF)$ and $\rho':U'\rightarrow\g\Diff^{\loc}_{0}(\RB^{q})$ be plots.  We must show that $A\circ(\rho\times\rho')$ is a plot of $\DS_{\g}(M,\FF)\times\DS_{\g}(M,\FF)$.  For each pair $(u,u')\in U\times U'$, there exist open neighbourhoods $V$ and $V'$ of $u$ and $u'$ respectively in $U$ and $U'$, and plots $\tilde{x}:V\rightarrow M$, $\tilde{f}:V\rightarrow\DS(M,\FF)$ and $\tilde{\varphi}:V'\rightarrow\Diff^{\loc}_{0}(\RB^{q})$ such that
	\[
	\rho(v) = (\tilde{x}(v),[\tilde{f}(v)]_{\tilde{x}(v)}),\hspace{7mm}\rho'(v') = [\tilde{\varphi}(v')]_{0}
	\]
	for all $(v,v')\in V\times V'$.  We compute
	\[
	(A\circ(\rho\times\rho'))(v,v') = \big((\tilde{x}(v),[\tilde{f}(v)]_{\tilde{x}(v)}),(\tilde{x}(v),[\tilde{\varphi}(v')^{-1}\circ \tilde{f}(v)]_{x})\big)
	\]
	for all $(v,v')\in V\times V'$, from which it is clear that $A\circ(\rho\times\rho')$ is a plot as required.
\end{proof}

\begin{defn}
	Let $(M,\FF)$ be a foliated manifold of codimension $q$.  We refer to the diffeological principal $\g\Diff^{\loc}_{0}(\RB^{q})$-bundle $\pi_{\DS_{\g}(M,\FF)}:\DS_{\g}(M,\FF)\rightarrow M$ as the \textbf{bundle of germs of distinguished functions on $(M,\FF)$}.
\end{defn}

The bundle of germs of distinguished functions on $(M,\FF)$ carries a canonical principal partial connection (Definition \ref{parcon}) induced by $\FF$.  For $(x,[f]_{x})$ denote by $H_{(x,[f]_{x})}$ the subspace of $T_{(x,[f]_{x})}\DS_{\g}(M,\FF)$ that consists of vectors of the form
\begin{equation}\label{hvec}
\rho(f,\gamma)_{*}(\partial_{t}),
\end{equation}
where $f$ is some representative of $[f]_{x}$, $\gamma:(-\epsilon,\epsilon)\rightarrow M$ is any smooth path contained in a leaf of $\FF$, with $\gamma(0) = x$ and with image contained in $\dom(f)$, and where $\rho(f,\gamma):(-\epsilon,\epsilon)\rightarrow\DS_{\g}(M,\FF)$ is the plot defined by
\[
\rho(f,\gamma)(t):=(\gamma(t),[f]_{\gamma(t)}),\hspace{7mm}t\in(-\epsilon,\epsilon).
\]
Notice that if $f'$ is any other representative of $[f]_{x}$, then by restricting the domain of $\gamma$ if necessary we have $\rho(f,\gamma) = \rho(f',\gamma)$.  Finally we define the diffeological subspace
\[
H:=\bigsqcup_{(x,[f]_{x})\in\DS_{\g}(M,\FF)}H_{(x,[f]_{x})}
\]
of $T\DS_{\g}(M,\FF)$.  The projection $\pi_{T\DS_{\g}(M,\FF)}$ restricts to $H$ to make $H$ a diffeological subbundle of $T\DS_{\g}(M,\FF)$.

\begin{prop}\label{Hcon}
	Let $(M,\FF)$ be a foliated manifold of codimension $q$.  Then $H$ is a principal partial connection on the on the principal $\g\Diff^{\loc}_{0}(\RB^{q})$-bundle $\pi_{\DS_{\g}(M,\FF)}:\DS_{\g}(M,\FF)\rightarrow M$.
\end{prop}

\begin{proof}
	If $(x,[f]_{x})\in\DS_{\g}(M,\FF)$, and $\rho(f,\gamma_{1})_{*}(\partial_{t})$ and $\rho(f,\gamma_{2})_{*}(\partial_{t})$ are two elements of $H_{(x,[f]_{x})}$ defined as in Equation \eqref{hvec}, then their images under $d\pi_{\DS_{\g}(M,\FF)}$ are equal if and only if $(\gamma_{1})_{*}(\partial_{t}) = (\gamma_{2})_{*}(\partial_{t})$ in $T_{x}\FF$.  Fibrewise injectivity of $d\pi_{\DS_{\g}(M,\FF)}|_{H}$ follows.
	
	It remains only to show that the right action of $\g\Diff^{\loc}_{0}(\RB^{q})$ sends $H$ into itself.  Suppose $(x,[f]_{x})\in\DS_{\g}(M,\FF)$, and choose a representative $f$ of $[f]_{x}$ and a smooth leafwise path $\gamma:(-\epsilon,\epsilon)\rightarrow M$ contained in $\dom(f)$ with $\gamma(0) = x$, so that $\rho(f,\gamma)_{*}(\partial_{t})\in H_{(x,[y_{\alpha}]_{x})}$.  For $[\varphi]_{0}\in\Diff^{\loc}_{0}(\RB^{q})$ denote by $R_{[\varphi]_{0}}$ the right action of $[\varphi]_{0}\in\g\Diff^{\loc}_{0}(\RB^{q})$ on $\DS_{\g}(M,\FF)$.  Then we see that 
	\[
	(dR_{[\varphi]_{0}})_{(x,[f]_{x})}\big(\rho(f,\gamma)_{*}(\partial_{t})\big) = \rho(\varphi^{-1}\circ f,\gamma)_{*}(\partial_{t})
	\]
	is the element of $H_{R_{[\varphi]_{0}}(x,[f]_{x})} = H_{(x,[\varphi^{-1}\circ f]_{x})}$ determined by the path $\gamma$.  Thus the action of $\g\Diff^{\loc}_{0}(\RB^{q})$ preserves $H$.
\end{proof}

\subsection{Parallel transport in the bundle of germs}

The partial connection $H$ can be used to lift leafwise paths in $M$ to $H$-wise paths in the total space $\DS_{\g}(M,\FF)$ in essentially the classical fashion.

\begin{thm}\label{parallel}
	Let $(M,\FF)$ be a foliated manifold of codimension $q$, and let $(\gamma,d)\in\PP_{T\FF}(M)$.  Denote $x:=\gamma(0)$. Then for each $(x,[f]_{x})\in\DS_{\g}(M,\FF)_{\gamma(0)}$, there exists a unique element smooth map $\gamma_{[f]_{x}}:\RB_{+}\rightarrow\DS_{\g}(M,\FF)$ such that:
	\begin{enumerate}
		\item $\gamma_{[f]_{x}}(0) = (x,[f]_{x})$,
		\item $\pi_{\DS_{\g}(M,\FF)}(\gamma_{[f]_{x}}(t)) = \gamma(t)$ for all $t\in\RB_{+}$, and
		\item $d\gamma_{[f]_{x}}:T\RB_{+}\rightarrow T\DS_{\g}(M,\FF)$ takes values in $H$.
	\end{enumerate}
\end{thm}

To prove this theorem we require the following lemma, which characterises solutions to the problem posed by Theorem \ref{parallel}.

\begin{lemma}\label{formlem}
	Let $(M,\FF)$ be a foliated manifold of codimension $q$, and let $(\gamma,d)\in\PP_{T\FF}(M)$.  A map $\gamma_{[f]_{x}}:\RB_{+}\rightarrow\DS_{\g}(M,\FF)$ satisfies items 1., 2. and 3. of Theorem \ref{parallel} if and only there exists a smooth function $\tilde{f}:\RB_{+}\rightarrow\DS(M,\FF)$, for which
	\begin{equation}\label{form}
	\gamma_{[f]_{x}}(t) = (\gamma(t),[\tilde{f}(t)]_{\gamma(t)}),\hspace{7mm}t\in\RB_{+},
	\end{equation}
	and such that for each $t\in\RB_{+}$, there is a neighbourhood $V_{t}$ of $t$ and $f_{t}\in\DS(M,\FF)$ such that $\gamma(s)\in\dom(f_{t})\subset\dom(\tilde{f}(s))$ and $\tilde{f}(s)|_{\dom(f_{t})} = f_{t}$ for all $s\in V_{t}$, where in particular $[f_{0}]_{x} = [f]_{x}$.
\end{lemma}

\begin{proof}
	Let us recall \cite{diffeology} that for any diffeological space $X$, a map $\psi:\RB_{+}\rightarrow X$ is smooth with respect to the subspace diffeology on $\RB_{+}\subset\RB$ (that is, with respect to the manifold-with-boundary structure on $\RB_{+}$) if and only if it admits an extension to a plot $\psi:U\rightarrow X$ defined on some open superset $U:=(-\epsilon,\infty)$ of $\RB_{+}$.  Since $\gamma$ has sitting instants, we can canonically extend it to a smooth leafwise path $\gamma:U\rightarrow M$ for which $\gamma|_{(-\epsilon,0]}\equiv\gamma(0)$.  Then to prove the lemma it is necessary and sufficient to prove that a plot $\gamma_{[f]_{x}}:U\rightarrow\DS_{\g}(M,\FF)$ satisfies items 1., 2., and 3. of Theorem \ref{parallel} on the larger domain $U$ if and only if there exists a plot $\tilde{f}:U\rightarrow\DS(M,\FF)$ satisfying the requirements listed in the statement of the lemma on the larger domain $U$.
	
	Suppose first that such a $\tilde{f}:U\rightarrow\DS(M,\FF)$ exists, and that $\gamma_{[f]_{x}}:U\rightarrow\DS_{\g}(M,\FF)$ is of the form given in Equation \eqref{form} on $U$.  That $\gamma_{[f]_{x}}$ satisfies items 1. and 2. of Theorem \ref{parallel} on $U$ is clear.  Note moreover that the properties of $\tilde{f}$ imply that for each $t\in U$, the function $\gamma_{[f]_{x}}|_{V_{t}}$ is equal to the plot $\rho(f_{t},\gamma|_{V_{t}}):V_{t}\rightarrow\DS_{\g}(M,\FF)$.  It follows immediately that $d\gamma_{[f]_{x}}$ takes values in $H$.
	
	Now suppose that $\gamma_{[f]_{x}}:U\rightarrow \DS_{\g}(M,\FF)$ is any smooth map satisfying items 1., 2. and 3. given in the statement.  Connectedness of $U$ together with items 1. and 2. of the statement imply that there is a smooth map $\tilde{f}:U\rightarrow \DS(M,\FF)$ such that
	\[
	\gamma_{[f]_{x}}(t) = (\gamma(t),[\tilde{f}(t)]_{\gamma(t)}),\hspace{7mm}t\in U.
	\]
	By item 3., and by definition of $H$, for each $t\in U$, there is a smooth leafwise path $\gamma_{t}$ contained in the domain of some $f_{t}\in\DS(M,\FF)$ such that, letting $\tau_{t}:s\mapsto s+t$ denote the translation map, we have
	\[
	(\gamma_{[f]_{x}}\circ\tau_{t})_{*}(\partial_{s}) = \rho(f_{t},\gamma_{t})_{*}(\partial_{s}).
	\]
	Now for $s$ near $t$, we have $\gamma_{[f]_{x}}\circ\tau_{t}(s) = (\gamma(t+s),[\tilde{f}(t+s)]_{\gamma(t+s)})$, while $\rho(f_{t},\gamma_{t})(s) = (\gamma_{t}(s),[f^{c}_{t}(s)]_{\gamma_{t}(s)})$, where $f_{t}^{c}$ is the constant plot $s\mapsto f_{t}$.  It follows then by Equation \eqref{equiv} that there is some diffeomorphism $\phi\in\Diff^{\loc}_{0}(\RB)$ and a neighbourhood $V_{t}$ of $t$ in $U$ for which we have
	\[
	\gamma(s) = \gamma_{t}\circ\phi\circ\tau_{-t}(s),\hspace{7mm} q\circ\tilde{f}(s) = q\circ f^{c}_{t}\circ\phi\circ\tau_{-t}(s) = q\circ f^{c}_{t}(s)
	\]
	for all $s\in V_{t}$.  Choosing $\dom(f_{t})$ sufficiently small, we then have $\gamma(s)\in\dom(f_{t})\subset\dom(\tilde{f}(s))$ and $\tilde{f}(s)|_{\dom(f_{t})} = f_{t}$ for all $s\in V_{t}$ as required.
\end{proof}

\begin{proof}[Proof of Theorem \ref{parallel}]
	We prove uniqueness by showing that any two paths of the form given in Equation \eqref{form} must be equal.  Indeed, suppose we are given two functions $\tilde{f}^{i}:\RB_{+}\rightarrow\DS(M,\FF)$, $i = 1,2$, with the properties given in the statement of Lemma \ref{formlem}.  For each $i = 1,2$ and each $t\in \RB_{+}$, let $V_{i}$ be the neighbourhood of $t$ and let $f_{t}^{i}\in\DS(M,\FF)$ be the function for which $\gamma(s)\in\dom(f_{t}^{i})\subset\dom(\tilde{f}^{i}(s))$ for which $\tilde{f}^{i}(s)|_{\dom(f^{i}_{t})} = f^{i}_{t}$.  We then have $[f^{1}_{0}]_{x} = [f]_{x} = [f^{2}_{0}]_{x}$, and moreover we therefore also have
	\[
	[\tilde{f}^{1}(t)]_{\gamma(t)} = [f^{1}_{0}]_{\gamma(t)} = [f^{2}_{0}]_{\gamma(t)} = [\tilde{f}^{2}(t)]_{\gamma(t)}
	\]
	for all $t$ contained in some neighbourhood of $0$ in $\RB_{+}$.  Repeating this process finitely many times along the compact subinterval $[0,d]$ of $\RB_{+}$ shows that $[\tilde{f}^{1}(t)]_{\gamma(t)} = [\tilde{f}^{2}(t)]_{\gamma(t)}$ for all $t\in\RB_{+}$, hence the associated solutions $\gamma_{[f]_{x}}^{i}:t\mapsto(\gamma(t),[\tilde{f}^{i}(t)]_{\gamma(t)})$ must also agree.
	
	To see existence, cover the image of $\tilde{\gamma}$ by a chain $\{U_{0},\dots,U_{k}\}$ of foliated charts, with $x\in U_{0}$ and with $\gamma(d)\in U_{k}$, and let $t_{1}<\cdots< t_{k}$ be a partition of $[0,d]\subset\RB_{+}$ such that $\gamma(t_{i})\in U_{i-1}\cap U_{i}$ for all $i=1,\dots,k$.  For $t\in \RB_{+}$ we define
	\[
	\dom(\tilde{f}(t)):=\begin{cases}
					\dom(y_{0}) = U_{0}&\text{ if $0\leq t\leq t_{1}$}\\
					\dom(y_{0,1}\circ y_{1})\subset U_{1}&\text{ if $t_{1}\leq t\leq t_{2}$}\\
					\vdots\\
					\dom(y_{0,1}\circ\cdots\circ y_{k-1,k}\circ y_{k})\subset U_{k}&\text{ if $t_{k}\leq t<\infty$}
					\end{cases},
	\]
	and for $y\in \dom(\tilde{f}(t))$ we define
	\[
	\tilde{f}(t)(y):=\begin{cases}
				y_{0}(y)&\text{ if $0\leq t \leq t_{1}$}\\
				y_{0,1}\circ y_{1}(y)&\text{ if $t_{1}\leq t\leq t_{2}$}\\
				\vdots\\
				y_{0,1}\circ\cdots\circ y_{k-1,k}\circ y_{k}(y)&\text{ if $t_{k}\leq t<\infty$}
				\end{cases}.
	\]
	Clearly then $\tilde{f}:\RB_{+}\rightarrow\DS(M,\FF)$ satisfies the properties required in Lemma \ref{formlem} for the associated function $\gamma_{[f]_{x}}:t\mapsto(\gamma(t),[\tilde{f}(t)]_{\gamma(t)})$ to be a solution.
\end{proof}

Recall now from Definition \ref{Hpath} that we can associate to the partial connection $H$ on $\DS_{\g}(M,\FF)$ the $H$-path category $\PP_{H}(\DS_{\g}(M,\FF))$, consisting of those elements of $\PP(\DS_{\g}(M,\FF))$ whose tangents are contained in $H$.  Theorem \ref{parallel} then effectively says that we have a lifting map $\PP_{T\FF}(M)\times_{s,\pi_{\DS_{\g}(M,\FF)}}\DS_{\g}(M,\FF)\rightarrow \PP_{H}(\DS_{\g}(M,\FF))$.

\begin{thm}\label{lifting}
	The lifting map $L_{\FF}:\PP_{T\FF}(M)\times_{s,\pi_{\DS_{\g}(M,\FF)}}\DS_{\g}(M,\FF)\rightarrow \PP_{H}(\DS_{\g}(M,\FF))$ defined by
	\[
	L_{\FF}\big((\gamma,d),(x,[f]_{x})\big):=(\gamma_{[f]_{x}},d),
	\]
	where $\gamma_{[f]_{x}}$ is given by Theorem \ref{parallel}, is smooth.
\end{thm}

\begin{proof}
	Let $p:U\rightarrow\PP_{T\FF}(M)$ and $\rho:V\rightarrow \DS_{\g}(M,\FF)$ be plots.  We must show that the map
	\begin{equation}\label{Lf}
	U\times_{s\circ p,\pi_{\DS_{\g}(M,\FF)}\circ\rho}V\ni(u,v)\mapsto L_{\FF}(p(u),\rho(v))\in\PP_{H}(\DS_{\g}(M,\FF))
	\end{equation}
	is smooth.  Write $p(u) = (\tilde{\gamma}(u),\tilde{d}(u))$ for $u\in U$.  Taking $V$ to be sufficiently small, we may assume that there are plots $\tilde{x}:V\rightarrow M$ and $\tilde{f}:V\rightarrow \DS(M,\FF)$ such that $\rho(v) = (\tilde{x}(v),[\tilde{f}(v)]_{\tilde{x}(v)})$ for all $v\in V$.
	
	Fix $u_{0}$ in $U$, and let $U_{0},\dots,U_{k}$ be a chain of foliated charts covering the image of $\tilde{\gamma}(u_{0})$ such that $U_{0}$ contains $\tilde{\gamma}(u_{0})(0)$ and such that $y_{0}(\tilde{\gamma}(u_{0})(0)) = 0$.  Then there is a neighbourhood $U'$ of $u_{0}$ in $U$ such that for all $u\in U'$, the image of $\tilde{\gamma}(u)$ is also contained in the chain of charts $U_{0},\dots,U_{k}$.  For each $u\in U'$ we let $\tilde{\tau}(u)\in\Diff(\RB^{q})$ be the translation operator $\vec{x}\mapsto\vec{x}-y_{0}(\tilde{\gamma}(u)(0))$ on $\RB^{q}$, so that $\tilde{\tau}:U'\rightarrow\Diff(\RB^{q})$ is a plot with respect to the functional diffeology on $\Diff(\RB^{q})$.  We thereby obtain new transverse coordinates $\tilde{y}_{0}(u)$ on $U_{0}$ defined for each $u\in U'$ by
	\[
	\tilde{y}_{0}(u):=\tilde{\tau}(u)\circ y_{0}.
	\]
	The resulting map $\tilde{y}_{0}:U'\rightarrow\DS(M,\FF)$ is a plot of $\DS(M,\FF)$.
	
	Now for each $(u,v)\in U'\times_{s\circ p,\pi_{\DS_{\g}(M,\FF)}\circ\rho}V$, there is an element $\tilde{\varphi}(u,v)\in\Diff^{\loc}_{0}(\RB^{q})$ such that $\tilde{f}(v) = \tilde{\varphi}(u,v)^{-1}\circ\tilde{y}_{0}(u)$.  Smoothness of $\tilde{f}$ and $\tilde{y}_{0}$ guarantees that $\tilde{\varphi}:U'\times_{s\circ p,\pi_{\DS_{\g}(M,\FF)}\circ\rho}V\rightarrow \Diff^{\loc}_{0}(\RB^{q})$ is smooth also.  We now define a smooth map $\tilde{g}:U'\times_{s\circ p,\pi_{\DS_{\g}(M,\FF)}\circ\rho}V\times\RB_{+}\rightarrow\DS(M,\FF)$ in a similar manner to the construction used in proving existence for Theorem \ref{parallel}.  More precisely, we choose a family $\tilde{t}_{1},\dots,\tilde{t}_{k}:U'\rightarrow\RB$ of smooth maps such that
	\[
	0<\tilde{t}_{1}(u)<\cdots<\tilde{t}_{k}(u) < d(u)
	\]
	for all $u\in U'$, and such that $\tilde{\gamma}(u)(\tilde{t}_{i}(u))\in U_{i-1}\cap U_{i}$ for all $i = 1,\dots,k$.  For $(u,v,t)\in U'\times_{s\circ p,\pi_{\DS_{\g}(M,\FF)}}V\times\RB_{+}$ we define
	\[
	\dom\big(\tilde{g}(u,v,t)\big):=\begin{cases}
									\dom(y_{0}) = U_{0}&\text{ if $0\leq t\leq t_{1}(u)$}\\
									\dom(y_{0,1}\circ y_{1})\subset U_{1}&\text{ if $t_{1}(u)\leq t\leq t_{2}(u)$}\\
									\vdots\\
									\dom(y_{0,1}\circ\cdots\circ y_{k-1,k}\circ y_{k})\subset U_{k}&\text{ if $t_{k}(u)\leq t<\infty$}
									\end{cases},
	\]
	and then for $y\in\dom(\tilde{g}(u,v,t))$ we define
	\[
	\tilde{g}(u,v,t)(y):=\begin{cases}
							\tilde{\varphi}(u,v)^{-1}\circ\tilde{\tau}(u)\circ y_{0}(y)&\text{ if $0\leq t\leq t_{1}(u)$}\\
							\tilde{\varphi}(u,v)^{-1}\circ\tilde{\tau}(u)\circ y_{0,1}\circ y_{1}(y)&\text{ if $t_{1}(u)\leq t\leq t_{2}(u)$}\\
							\vdots\\
							\tilde{\varphi}(u,v)^{-1}\circ\tilde{\tau}(u)\circ y_{0,1}\circ\cdots\circ y_{k-1,k}\circ y_{k}(y)&\text{ if $t_{k}(u)\leq t<\infty$}
							\end{cases}.
	\]
	Then $\tilde{g}:U'\times_{s\circ p,\pi_{\DS_{\g}(M,\FF)}\circ\rho}V\times\RB_{+}\rightarrow\DS(M,\FF)$ is a smooth map, and for each fixed $(u,v)\in U'\times_{s\circ p,\pi_{\DS_{\g}(M,\FF)}\circ\rho}V$ the function $\tilde{g}(u,v,\cdot):\RB_{+}\rightarrow\DS(M,\FF)$ satisfies the property described in Lemma \ref{formlem}.  Moreover by construction we have
	\[
	L_{\FF}(p(u),\rho(v))(t) = (\tilde{\gamma}(u)(t),[\tilde{g}(u,v,t)]_{\tilde{\gamma}(t)})
	\]
	for all $(u,v,t)\in U'\times_{s\circ p,\pi_{\DS_{\g}(M,\FF)}\circ\rho}V\times\RB_{+}$, from which it follows by the smoothness of $\tilde{g}$ and $\tilde{\gamma}$ that the map given in Equation \eqref{Lf} is smooth.  Hence $L_{\FF}$ is smooth.
\end{proof}

Using Theorem \ref{lifting}, we can prove that the parallel transport map given by Theorem \ref{parallel} gives rise to define a canonical leafwise transport functor.

\begin{thm}\label{Tf}
	Let $(M,\FF)$ be a foliated manifold of codimension $q$.  Then the formula
	\[
	T_{\FF}(\gamma,d)(\gamma(0),[f]_{\gamma(0)}):=\gamma_{[f]_{\gamma(0)}}(d),\hspace{7mm}(\gamma,d)\in\PP_{T\FF}(M),\,(\gamma(0),[f]_{\gamma(0)})\in\DS_{\g}(M,\FF)
	\]
	defines a leafwise transport functor $T_{\FF}:\PP_{T\FF}(M)\rightarrow\Aut(\pi_{\DS_{\g}(M,\FF)})$.
\end{thm}

\begin{proof}
	That $T_{\FF}$ is a functor follows immediately from item 2. and uniqueness in Theorem \ref{parallel}.  Thus we need only show that it is smooth.  However this follows from Lemma \ref{lifting}.  Indeed, the functor $T_{\FF}$ is equal to the composite
	\[
	T_{\FF} = r\circ L_{\FF},
	\]
	where $L_{\FF}:\PP_{T\FF}(M)\times_{s,\pi_{\DS_{\g}(M,\FF)}}\DS_{\g}(M,\FF)\rightarrow\PP_{H}(\DS_{\g}(M\FF))$ is the smooth lifting map of Lemma \ref{lifting}, and where $r:\PP_{H}(\DS_{\g}(M,\FF))\rightarrow \DS_{\g}(M,\FF)$ is simply the smooth range map $(\gamma,d)\mapsto\gamma(d)$.  As the composite of two smooth maps, $T_{\FF}$ is itself smooth.
\end{proof}

\begin{rmk}\label{wink1}\normalfont
	The reader who is familiar with J. Phillips' paper \cite{holimper} will notice that our leafwise transport functor given in Theorem \ref{Tf} is essentially a categorification of Phillips' \cite[p. 160, Proposition]{holimper} using the language of diffeology.  Indeed, Phillips' work is the central inspiration for our constructions.
	
	Let us briefly discuss how our constructions relate to those of Winkelnkemper by comparing the functor of Theorem \ref{Tf} with Winkelnkemper's holonomy diffeomorphisms (see Proposition \ref{holdiff}).  The key point is that, given $(\gamma,d)\in\PP_{T\FF}(M)$ and a germ $[f]_{x}$ of a distinguished function at $x = s(\gamma,d)$, the element $T_{\FF}(\gamma,d)\big(x,[f]_{x}\big)$ of $\DS_{\g}(M,\FF)_{y = r(\gamma,d)}$ is determined uniquely by the input data, while on the other hand Winkelnkemper requires a \emph{choice} of germ $[g]_{y}$ which is related to $[f]_{x}$ by the germ of the holonomy diffeomorphism $H_{\gamma}^{\UU}$ associated to $\gamma$ and to any chain $\UU$ of charts whose initial and terminal transverse coordinate maps are $f$ and $g$ respectively.  By maximality of the foliated atlas, we can always find an element $[\varphi]_{0}\in\Diff_{0}^{\loc}(\RB^{q})$ such that $(y,[\varphi]^{-1}_{0} [g]_{y}) = T_{\FF}(\gamma,d)\big(x,[f]_{x}\big)$. Thus letting $\cdot\varphi$ be the map $[g]_{y}\mapsto[\varphi]_{0}^{-1}[g]_{y}$, and letting $T_{W}(\gamma)$ be the map $[f]_{x}\mapsto [f]_{x}[(H^{\UU}_{\gamma})^{-1}]_{y}$, then the diagram
	\begin{center}
		\begin{tikzcd}
			\DS_{\g}(M,\FF)_{x} \ar[rr,"T_{W}(\gamma)"] \ar[d,"\id"] & & \DS_{\g}(M,\FF)_{y} \ar[d,"\cdot\varphi"] \\ \DS_{\g}(M,\FF)_{x} \ar[rr,"T_{\FF}(\gamma)"] & & \DS_{\g}(M,\FF)_{y}
		\end{tikzcd}
	\end{center}
	commutes.  Thus Winkelnkemper's ``holonomy diffeomorphism" differs from a true parallel transport map by a gauge transformation that is determined by the choice of chain used to define the diffeomorphism.  Of course by Proposition \ref{holdiff}, this choice does not affect the holonomy class of $\gamma$.
\end{rmk}

The diffeological groupoid $\HH(T_{\FF})$ associated to the transport functor $T_{\FF}:\PP_{T\FF}(M)\rightarrow\Aut(\pi_{\DS_{\g}(M,\FF)})$ is isomorphic, as a diffeological groupoid, to the Winkelnkemper-Phillips holonomy groupoid.

\begin{thm}\label{coincide}
	Let $(M,\FF)$ be a foliated manifold, let $\HH(M,\FF)$ be its Winkelnkemper-Phillips holonomy groupoid, and let $\HH(T_{\FF})$ be the holonomy groupoid associated to the leafwise transport functor $T_{\FF}$ of Theorem \ref{Tf}.  Then the map $\iota:\HH(T_{\FF})\rightarrow\HH(M,\FF)$ defined by
	\[
	\iota([(\gamma,d)]):=(\gamma(d),\gamma(0),[\gamma])
	\]
	is an isomorphism of diffeological groupoids.
\end{thm}

\begin{proof}
	Our discussion in Remark \ref{wink1} shows that two paths are equivalent in the sense of Winkelnkemper if and only if they are $T_{\FF}$-equivalent, and it follows immediately that the identification $[(\gamma,d)]\mapsto(\gamma(d),\gamma(0),[\gamma])$ is an isomorphism of groupoids.  We need only show that it is smooth with smooth inverse.
	
	Suppose first that we have a plot $\rho:U\rightarrow \HH(T_{\FF})$.  We must show that the composite $\iota\circ\rho$ is a smooth function.  Fixing $u_{0}\in U$, there is an open neighbourhood $V$ of $u_{0}$ such that $\iota\circ\rho(V)$ is contained in the domain of some coordinate neighbourhood $(V_{(y,x,\gamma,\UU)},\psi)$ of $\HH(M,\FF)$ (see the construction immediately preceding Definition \ref{winkgpd}), and we then need to show that $\psi\circ\iota\circ\rho|_{V}$ is smooth as a map between Euclidean spaces.  Let $(U_{0},x_{0},y_{0})$ and $(U_{k},x_{k},y_{k})$ be the initial and terminal charts of $\UU$ respectively.  Without loss of generality we can assume that there is a plot $\tilde{\rho}:V\rightarrow\PP_{T\FF}(M)$ such that $\rho|_{V} = q\circ\tilde{\rho}$, where $q:\PP_{T\FF}(M)\rightarrow\HH(T_{\FF})$ denotes the quotient map.  Smoothness of the map $(t,v)\mapsto\tilde{\rho}(v)(t)$ then implies that the maps $v\mapsto r(\tilde{\rho}(v))$ and $v\mapsto s(\tilde{\rho}(v))$ are smooth, hence
	\[
	\psi\circ\iota\circ\rho:v\mapsto\big(x_{k}(r(\tilde{\rho}(v))), x_{0}(s(\tilde{\rho}(v))), y_{0}(s(\tilde{\rho}(v)))\big)
	\]
	is smooth also.
	
	Now suppose that $\rho':U\rightarrow\HH(M,\FF)$ is a plot.  Fixing $u_{0}\in U$ we must find a neighbourhood $V$ of $u_{0}$ and a plot $\tilde{\rho}':V\rightarrow\PP_{T\FF}(M)$ such that $q\circ\tilde{\rho}' = \iota^{-1}\circ\rho'$.  Let $V$ be any neighbourhood of $u_{0}$ such that $\rho'(V)$ is contained in a coordinate neighbourhood $(V_{(y,x,\gamma,\UU)},\psi)$ of $\HH(M,\FF)$, and let the initial and terminal plaques of $\UU$ be denoted $(U_{0},x_{0},y_{0})$ and $(U_{k},x_{k},y_{k})$ respectively.  Let $(a,b,c)$ be the smooth functions on $V$ defined by $(a(v),b(v),c(v)) = \psi\circ\rho'(v)$ for all $v\in V$.  There are now four functions from $V$ into $M$ that are of interest to us.  These are
	\[
	x(v):=\psi_{0}^{-1}(b(v),c(v)),\hspace{3mm} z(v):=\psi_{0}^{-1}(x_{0}(\gamma(0)),c(v))
	\]
	in the plaque $P_{0,c(v)}$ determined by $c(v)$ in $U_{0}$, and 
	\[
	w(v):=H_{\gamma}^{\UU}(z(v)),\hspace{3mm} y(v):=\psi_{k}^{-1}(b(v),y_{k}(w(v)))
	\]
	in the plaque $P_{k,c(v)}$ determined by $c(v)$ in $U_{k}$.  For each $v\in V$ we define $\gamma_{z(v)}$ exactly as in Equation \eqref{gammaz}, we let $s_{x(v)}$ be the straight line path (with sitting instants) in $P_{0,c(v)}$ joining $x(v)$ to $z(v)$, and let $s_{y(v)}$ be the straight line path (with sitting instants) in $P_{k,c(v)}$ joining $w(v)$ to $y(v)$.  Then the map $\tilde{\rho}':V\rightarrow\PP_{T\FF}(M)$ which sends $v$ to the concatenation $s_{y(v)}\gamma_{z(v)}s_{x(v)}$ is a plot satisfying $q\circ\tilde{\rho}' = \iota^{-1}\circ\rho'$.
\end{proof}

As a consequence of Theorem \ref{coincide} we give the following definition.

\begin{defn}\label{holonomy}
	Let $(M,\FF)$ be a foliated manifold.  We refer to the diffeological groupoid $\HH(T_{\FF})$ associated to the leafwise transport functor $T_{\FF}$ as the \textbf{holonomy groupoid of $(M,\FF)$}, and denote it simply by $\HH(M,\FF)$.
\end{defn}

We will see in the next section how similar techniques can be used to define transport functors and holonomy groupoids for foliated bundles.

\section{The holonomy groupoid of a foliated bundle}

\subsection{The fibre holonomy groupoid of a foliated bundle}

In this subsection we consider a foliated bundle $\pi_{B}:B\rightarrow M$ (see Definition \ref{folbund}), and denote by $\FF$ the induced foliation of $M$.  Leafwise paths in $M$ can be lifted to leafwise paths in $B$ using the partial connection in $\pi_{B}$ in the classical fashion, and these liftings will define a canonical leafwise transport functor.  We recall the following simple application of the Picard-Lindel\"{o}f theorem, phrased in our diffeological categorical terminology.

\begin{thm}\cite[Theorem 9.8]{natopdiffgeom}\label{transport}
Let $\pi_{B}:B\rightarrow M$ be a foliated bundle over $M$, and let $(\gamma,d)\in\PP_{T\FF}(M)$.  To each $b\in B_{\gamma(0)}$, there exists a unique curve $\gamma_{b}:\RB_{+}\rightarrow B$, such that
\begin{enumerate}
	\item $\gamma_{b}(0) = b$,
	\item $\pi_{B}(\gamma_{b}(t)) = \gamma(t)$ for all $t\in\RB_{+}$, and
	\item $d\gamma_{b}:T\RB_{+}\rightarrow TB$ takes values in $T\FF_{B}$.
\end{enumerate}
If moreover $\gamma$ depends smoothly on some further parameters, then so too does $\gamma_{b}$.  That is, the lifting function $L_{B}^{0}:\PP_{T\FF}(M)\times_{s,\pi_{B}}B\rightarrow \PP_{T\FF_{B}}(B)$ defined by
\[
L_{B}^{0}\big((\gamma,d),b\big):=(\gamma_{b},d)
\]
is smooth.\qed
\end{thm}

Theorem \ref{transport} enables us to assign to each foliated bundle $\pi_{B}$ associated to $(M,\FF)$ a canonical leafwise transport functor $T:\PP_{T\FF}(M)\rightarrow\Aut(\pi_{B})$, which is defined simply via the classical lifting map arising from Theorem \ref{transport}.

\begin{thm}\label{bicboi}
Let $\pi_{B}:B\rightarrow M$ be a foliated bundle associated to $(M,\FF)$.  Then the formula
\[
T_{\pi_{B}^{0}}(\gamma,d)(b):=\gamma_{b}(d),\hspace{7mm}(\gamma,d)\in \PP_{T\FF}(M),\,b\in B_{\gamma(0)}
\]
defines a leafwise transport functor $T_{\pi_{B}^{0}}:\PP_{T\FF}(M)\rightarrow\Aut(\pi_{B})$.
\end{thm}

\begin{proof}
Functoriality and smoothness both follow from Theorem \ref{transport}.
\end{proof}

Theorem \ref{bicboi} allows us to define a new holonomy groupoid associated canonically to any foliated bundle.

\begin{defn}\label{fibreholonomy}
Let $\pi_{B}:B\rightarrow M$ be a foliated bundle associated to $(M,\FF)$.  The holonomy groupoid $\HH(T_{\pi_{B}^{0}})$ associated to the leafwise transport functor $T_{\pi_{B}^{0}}$ is called the \textbf{fibre holonomy groupoid associated to the foliated bundle $\pi_{B}$}.
\end{defn}

\begin{rmk}\normalfont\label{jets2}
	By the discussion of Remark \ref{jets}, each transverse jet bundle $\pi_{B}^{k}:J^{k}_{\t}(\pi_{B})\rightarrow M$ is also a foliated bundle, with foliation $\FF_{k}$ of its total space $J^{k}_{\t}(\pi_{B})$.  Thus  Theorems \ref{transport} and \ref{bicboi} may be applied to yield a smooth lifting function
	\[
	L_{B}^{k}:\PP_{T\FF}(M)\times_{s,\pi_{B}^{k}}J^{k}_{\t}(\pi_{B})\rightarrow\PP_{T\FF_{k}}(J^{k}_{\t}(\pi_{B})),
	\]
	and an associated leafwise transport functor
	\[
	T_{\pi_{B}^{k}}:\PP_{T\FF}(M)\rightarrow\Aut(\pi_{B}^{k})
	\]
	and holonomy groupoid $\HH(T_{\pi_{B}^{k}})$.
\end{rmk}

The fibre holonomy groupoid of a foliated bundle $\pi_{B}$ is, as we will see in the next section, the \emph{smallest} of a hierarchy of holonomy groupoids associated to $\pi_{B}$.  Its elements are easily accessible - to determine the $T_{\pi_{B}^{0}}$-equivalence class of a smooth path $\gamma$ one need only solve a relatively simple parallel transport differential equation as in Theorem \ref{transport}.

Thus $\HH(T_{\pi_{B}^{0}})$ may be thought of as capturing the holonomy of 0-jets of distinguished sections of $\pi_{B}$.  The higher fibre holonomy groupoids $\HH(T_{\pi_{B}^{k}})$ of the foliated jet bundles $\pi_{B}^{k}:J^{k}_{\t}(\pi_{B})\rightarrow M$ associated to $\pi_{B}$ (see Remark \ref{jets}) then capture the holonomy of $k$-jets of distinguished sections of $\pi_{B}$.  In analogy with the previous section, there is a bigger prize - namely a groupoid which captures the holonomy of the full \emph{germs} of distinguished sections of $\pi_{B}$.  Constructing such a groupoid is our objective in the next section.

\subsection{The holonomy groupoid of a foliated bundle}

The construction of the holonomy groupoid of a foliated manifold begun with the construction of a bundle of germs of distinguished functions.  In the same way, the construction of the holonomy groupoid of a foliated bundle begins with the construction of a bundle of germs of distinguished sections.

Let $\pi_{B}:B\rightarrow M$ be a foliated bundle, and denote by $q$ the codimension of the induced foliation on the base.  Recall (Definition \ref{distinguishedsec}) that a \emph{distinguished section} of $\pi_{B}$ is a section $\sigma$ of $\pi_{B}$, defined over some open neighbourhood $\dom(\sigma)$ of $M$, such that there is a distinguished function $f\in\DS(M,\FF)$ defined on $\dom(\sigma)$ and a smooth function $\sigma_{f}:\range(f)\rightarrow B$ such that
\[
\sigma = \sigma_{f}\circ f.
\]
Denote by $\DS(\pi_{B})$ the set of all distinguished sections of $\pi_{B}$, equipped with the subspace diffeology of the functional diffeology (see Definition \ref{functionaldiff2}) on $C^{\infty}_{\loc}(M,B)$.  Equip the set $M\times\DS(\pi_{B})$ with the product diffeology, and equip the subset
\[
S:=\{(x,\sigma)\in M\times\DS(\pi_{B}):x\in\dom(\sigma)\}
\]
of $M\times\DS(\pi_{B})$ with the resulting subspace diffeology.  Define an equivalence relation $\sim$ on $S$ by declaring $(x,\sigma)\sim(y,\sigma')$ if and only if $x =y$ and $[\sigma]_{x} = [\sigma']_{x}$.  Finally, consider the diffeological quotient
\[
\DS_{\g}(\pi_{B}):=S/\sim
\]
of $S$. Points of $\DS_{\g}(\pi_{B})$ can be written $(x,[\sigma]_{x})$, where $[\sigma]_{x}$ denotes an equivalence class of distinguished sections defined about $x$ under $\sim$.  In this quotient diffeology, a parametrisation $\rho:U\rightarrow\DS_{\g}(\pi_{B})$ is a plot if and only if for each $u_{0}\in U$, there is a neighbourhood $V$ of $u_{0}$ in $U$ and plots $\tilde{x}:V\rightarrow M$ and $\tilde{\sigma}:V\rightarrow\DS(\pi_{B})$ such that
\begin{equation}\label{rho}
\rho(u) = (\tilde{x}(u),[\tilde{\sigma}(u)]_{\tilde{x}(u)})
\end{equation}
for all $u\in V$.  The natural projection $\pi_{\DS_{\g}(\pi_{B})}:\DS_{\g}(\pi_{B})\rightarrow M$ defined by
\[
\pi_{\DS_{\g}(\pi_{B})}(x,[\sigma]_{x}):=x,\hspace{7mm}(x,[\sigma]_{x})\in\DS_{\g}(\pi_{B})
\]
is then clearly a subduction: any plot $x:U\rightarrow M$ of $M$ can be realised as the composite $\pi_{\DS_{\g}(\pi_{B})}\circ\rho$, where $\rho$ is some plot of the form given in Equation \eqref{rho} for some plot $\tilde{\rho}:U\rightarrow\DS(\pi_{B})$.  In fact, we can show using the local structure of the foliated bundle $\pi_{B}:B\rightarrow M$ that $\pi_{\DS_{\g}(\pi_{B})}:\DS_{\g}(\pi_{B})\rightarrow M$ is a diffeological fibre bundle.

\begin{lemma}
	Let $\pi_{B}:B\rightarrow M$ be a foliated bundle, and denote by $q$ the codimension of the induced foliation $\FF$ of $M$.  Then the characteristic map $(r,s):\Aut(\pi_{\DS_{\g}(\pi_{B})})\rightarrow M\times M$ is a subduction, hence $\pi_{\DS_{\g}(\pi_{B})}:\DS_{\g}(\pi_{B})\rightarrow M$ is a diffeological fibre bundle.
\end{lemma}

\begin{proof}
	Let $\tilde{x}:U\rightarrow M\times M$ be a plot, so that $\tilde{x}_{1}:=\proj_{1}\circ\tilde{x}$ and $\tilde{x}_{2}:=\proj_{2}\circ\tilde{x}$ are plots $U\rightarrow M$.  We must show that $\tilde{x}$ is locally of the form $(r,s)\circ\rho$ for some plot $\rho$ of $\Aut(\pi_{\DS_{\g}(\pi_{B})})$.
	
	For any $u_{0}\in U$, we can find a neighbourhood $V$ of $u_{0}$ in $U$ which is sufficiently small that $\tilde{x}_{i}(V)$ is contained in a foliated chart $(U_{i},x_{i},y_{i})$ of $M$ with local trivialisation $\tau_{i}:B|_{U_{i}}\rightarrow U_{i}\times F$ of $B$ for $i =1,2$.  Let us assume without loss of generality that $y_{1}(U_{1}) = y_{2}(U_{2})$.
	
	Now, for $v\in V$, any distinguished section $\sigma$ of $\pi_{B}$ defined over $U_{2}$ is associated to a unique $\eta:y_{2}(U_{2})\rightarrow F$ such that
	\[
	\sigma(x) = \tau_{2}^{-1}\big(x,\eta\circ y_{2}(x)\big)
	\]
	for all $x\in U_{2}$.  To this $\sigma$ we associate $\tilde{\sigma}$ defined on $U_{1}$ by
	\[
	\tilde{\sigma}(x):=\tau_{1}^{-1}\big(x,\eta\circ y_{1}(x)\big)
	\]
	for all $x\in U_{2}$.  We then define a parametrisation $\rho:V\rightarrow\Aut(\pi_{\DS_{\g}(\pi_{B})})$ by the formula
	\[
	\rho(u)\big(\tilde{x}_{2}(u),[\sigma]_{\tilde{x}_{2}(u)}\big):=\big(\tilde{x}_{1}(u),[\tilde{\sigma}]_{\tilde{x}_{1}(u)}\big),
	\]
	and it is evident that the associated maps defined as in Equations \eqref{f1}, \eqref{f2} and \eqref{f3} are smooth, implying that $\rho$ is a plot.  Now by definition $(r,s)\circ\rho = \tilde{x}|_{V}$, hence $(r,s)$ is a subduction as claimed.
\end{proof}

\begin{defn}
	Let $\pi_{B}:B\rightarrow M$ be a foliated bundle.  We refer to the diffeological fibre bundle $\pi_{\DS_{\g}(\pi_{B})}:\DS_{\g}(\pi_{B})\rightarrow M$ as the \textbf{bundle of germinal distinguished sections of $\pi_{B}$}.
\end{defn}

\begin{rmk}\label{assoc}\normalfont
	In fact the bundle of germinal distinguished sections can be equivalently thought of as a sort of associated bundle for $\DS_{\g}(M,\FF)$ in the following way. Denote by $\BS$ the diffeological subspace of $\in C^{\infty}_{\loc}(\RB^{q},B)$ consisting of functions $\eta$ that are defined in a neighbourhood of $0$ and for which $\pi_{B}\circ\eta:\dom(\eta)\rightarrow M$ is constant.  Denote by $\BS_{\g}$ the diffeological quotient of $\BS$ by the equivalence relation which deems $\eta$ and $\eta'$ equivalent if and only if their germs $[\eta]_{0}$ and $[\eta']_{0}$ at zero coincide.
	
	Then $\BS_{\g}$ carries a natural left action of $\g\Diff_{0}^{\loc}(\RB^{q})$ defined by composition of germs:
	\[
	[\varphi]_{0}\cdot[\eta]_{0}:=[\eta]_{0}[\varphi^{-1}]_{0},\hspace{7mm}[\varphi]_{0}\in\g\Diff_{0}^{\loc}(\RB^{q}),\,[\eta]_{0}\in\BS_{\g}.
	\]
	We consider the diffeological subspace $\DS_{g}(M,\FF)\times_{M}\BS_{\g}$ of the diffeological product $\DS_{\g}(M,\FF)\times\BS$, consisting of triples $(x,[f]_{x},[\eta]_{0})$ for which $\pi_{B}(\eta(0)) =x$.  Thereon, we have the diagonal action of $\g\Diff_{0}^{\loc}(\RB^{q})$ defined for $(x,[f]_{x},[\eta]_{0})\in\DS_{\g}(M,\FF)\times_{M}\BS_{\g}$ and $[\varphi]_{0}\in\g\Diff_{0}^{\loc}(\RB^{q})$ again by composition of germs:
	\[
	(x,[f]_{x},[\eta]_{0})\cdot[\varphi]_{0}:=(x,[\varphi^{-1}]_{0}[f]_{x},[\eta]_{0}[\varphi]_{0}).
	\]
	Finally the diffeological quotient
	\[
	(\DS_{\g}(M,\FF)\times_{M}\BS_{\g})/\g\Diff_{0}^{\loc}(\RB^{q})
	\]
	identifies naturally with the space $\DS_{\g}(\pi_{B})$.
	
	Indeed, for any $(x,[\sigma]_{x})\in\DS_{\g}(\pi_{B})$ and for any $(x,[f]_{x})\in\DS_{\g}(M,\FF)$, there is a unique $[\sigma_{f}]_{0}\in\BS_{\g}$ such that we may factorise $[\sigma]_{x} = [\sigma_{f}]_{0}[f]_{x}$.  Since moreover any other choice $(x,[g]_{x})\in\DS_{\g}(M,\FF)$ differs from $(x,[f]_{x})$ only by some element of $\g\Diff_{0}^{\loc}(\RB^{q})$, the map sending $(x,[f]_{x},[\sigma_{f}]_{0})\in\DS_{\g}(M,\FF)\times_{M}\BS_{\g}$ to $(x,[\sigma_{f}]_{0}[f]_{x})\in\DS_{\g}(\pi_{B})$ descends to a diffeomorphism $(\DS_{\g}(M,\FF)\times_{M}\BS_{\g})/\g\Diff_{0}(\RB^{q})\cong\DS_{\g}(\pi_{B})$.
\end{rmk}

We now construct a partial connection on $\DS_{\g}(\pi_{B})$. For $(x,[\sigma]_{x})\in\DS_{\g}(\pi_{B})$, one defines $(H_{B}^{\g})_{(x,[\sigma]_{x})}$ to be the subspace of $T_{(x,[\sigma]_{x})}\DS_{\g}(\pi_{B})$ consisting of vectors of the form
\[
\rho(\sigma,\gamma)_{*}(\partial_{t}),
\]
where $\sigma$ is any representative of $[\sigma]_{x}$, $\gamma:(-\epsilon,\epsilon)\rightarrow M$ is any smooth leafwise path in $\dom(\sigma)$ with $\gamma(0) = x$, and where $\rho(\sigma,\gamma):(-\epsilon,\epsilon)\rightarrow \DS_{\g}(\pi_{B})$ is the plot defined by
\[
\rho(\sigma,\gamma)(t):=(\gamma(t),[\sigma]_{\gamma(t)}),\hspace{7mm}t\in(-\epsilon,\epsilon).
\]
Similar arguments to those used in the proof of Proposition \ref{Hcon} then give the following result.

\begin{prop}
	Let $\pi_{B}:B\rightarrow M$ be a foliated bundle.  Then the subbundle
	\[
	H_{B}^{\g}:=\bigsqcup_{(x,[\sigma]_{x})}(H_{B}^{\g})_{(x,[\sigma]_{x})}
	\]
	of $T\DS_{\g}(\pi_{B})$ is a partial connection for the diffeological bundle $\pi_{\DS_{\g}(\pi_{B})}:\DS_{g}(\pi_{B})\rightarrow M$.\qed
\end{prop}

The following lifting result follows in a similar fashion to Theorem \ref{parallel} and Theorem \ref{lifting}.

\begin{thm}\label{liftingB}
	Let $\pi_{B}:B\rightarrow M$ be a foliated bundle associated to a foliation $\FF$ of $M$.  Then for each $\big((\gamma,d),(x,[\sigma]_{x})\big)\in\PP_{T\FF}(M)\times_{s,\pi_{\DS_{\g}(\pi_{B})}}\DS_{\g}(\pi_{B})$ there is a unique $(\gamma_{[\sigma]_{x}},d)\in\PP_{H^{\g}_{B}}(\DS_{\g}(\pi_{B}))$ with $\gamma_{[\sigma]_{x}}(0) = (x,[\sigma]_{x})$ and such that $\pi_{\DS_{\g}(\pi_{B})}\circ\gamma_{[\sigma]_{x}} = \gamma$.  The resulting lifting map $L_{B}^{\g}:\PP_{T\FF}(M)\times_{s,\pi_{\DS_{\g}(\pi_{B})}}\DS_{\g}(\pi_{B})\rightarrow\PP_{H^{\g}_{B}}(\DS_{\g}(\pi_{B}))$ given by
	\[
	L_{B}^{\g}\big((\gamma,d),(x,[\sigma]_{x})\big):=(\gamma_{[\sigma]_{x}},d)
	\]
	is smooth.
\end{thm}

\begin{proof}
	To show existence and uniqueness one replaces Lemma \ref{formlem} as follows.  Similar arguments to those used in Lemma \ref{formlem} show that a map $\gamma_{[\sigma]_{x}}:\RB_{+}\rightarrow\DS_{\g}(M,\FF)$ satisfies the requirements of Theorem \ref{liftingB} if and only if there exists a smooth function $\tilde{\sigma}:\RB_{+}\rightarrow\DS(\pi_{B})$ for which
	\[
	\gamma_{[\sigma]_{x}}(t) = (\gamma(t),[\tilde{\sigma}(t)]_{\gamma(t)}),\hspace{7mm}t\in\RB_{+},
	\]
	and such that for each $t\in\RB_{+}$, there exists a neighbourhood $V_{t}$ of $t$ and $\sigma_{t}\in\DS(\pi_{B})$ such that $\gamma(s)\in\dom(\sigma_{t})\subset\dom(\tilde{\sigma}(s))$ and $\tilde{\sigma}(s)|_{\dom(\sigma_{t})} = \sigma_{t}$ for all $s\in V_{t}$, where in particular $[\sigma_{0}]_{x} = [\sigma]_{x}$.  Now uniqueness follows in essentially the same way as the uniqueness part of Theorem \ref{parallel}.  For existence, one chooses a chain $\{U_{0},\dots,U_{k}\}$ of foliated charts covering $\range(\gamma)$ and associated partition $t_{1}<\cdots<t_{k}$ of $[0,d]$, and insists \emph{additionally} that the $U_{i}$ are associated to local trivialisations $\tau_{i}:B|_{U_{i}}\rightarrow U_{i}\times F$.  Letting $\sigma_{0}:\range(y_{0})\rightarrow F$ be such that the composite $\tau_{0}^{-1}\circ\big(\id_{U_{0}}\times(\sigma_{0}\circ y_{0})\big)$ represents $[\sigma]_{x}$, one then defines
	\[
	\dom(\tilde{\sigma}(t)):=\begin{cases}
							\dom(y_{0}) = U_{0}&\text{ if $0\leq t\leq t_{1}$}\\
							\dom(\tau_{1,0}\circ\sigma_{0}\circ y_{0,1}\circ y_{1})\subset U_{1}&\text{ if $t_{1}\leq t\leq t_{2}$}\\
							\vdots\\
							\dom(\tau_{k,k-1}\circ\cdots\circ \tau_{1,0}\circ\sigma_{0}\circ y_{0,1}\circ\cdots\circ y_{k-1,k}\circ y_{k})\subset U_{k}&\text{ if $t_{k}\leq t<\infty$}
							\end{cases}
	\]
	and
	\[
	\tilde{\sigma}(t)(y):=\begin{cases}
							\tau_{0}^{-1}\big(y,\sigma_{0}\circ y_{0}(y)\big)&\text{ if $0\leq t\leq t_{1}$}\\
							\tau_{1}^{-1}\big(y,\tau_{1,0}\circ\sigma_{0}\circ y_{0,1}\circ y_{1}(y)\big)&\text{ if $t_{1}\leq t\leq t_{2}$}\\
							\vdots\\
							\tau_{k}^{-1}\big(y,\tau_{k,k-1}\circ\cdots\circ\tau_{1,0}\circ\sigma_{0}\circ y_{0,1}\circ\cdots\circ y_{k-1,k}\circ y_{k}(y)\big)&\text{ if $t_{k}\leq t < \infty$}
							\end{cases}
	\]
	for all $t\in\RB_{+}$, and so obtains $\tilde{\sigma}:\RB_{+}\rightarrow\DS_{\g}(\pi_{B})$ with the required properties.
	
	Finally, taking into account these modifications to the proof of existence in Theorem \ref{parallel}, smoothness of the resulting lifting map $L_{B}^{\g}:\PP_{T\FF}(M)\times_{s,\pi_{\DS_{\g}(\pi_{B})}}\DS_{\g}(\pi_{B})\rightarrow\PP_{H_{B}^{\g}}(\DS_{\g}(\pi_{B}))$ follows from similar arguments to those of Theorem \ref{lifting}.
\end{proof}

The composition of $L_{B}^{\g}$ with the range map on $\PP_{H^{\g}_{B}}(\DS_{\g}(\pi_{B}))$ now yields a leafwise transport functor in essentially the same fashion as in Theorem \ref{Tf}.

\begin{thm}\label{TfB}
	Let $\pi_{B}:B\rightarrow M$ be a foliated bundle.  Then the formula
	\[
	T_{\pi_{B}^{\g}}(\gamma,d)\big(x,[\sigma]_{x}\big):=\gamma_{[\sigma]_{x}}(d),\hspace{7mm}\big((\gamma,d),(x,[\sigma]_{x})\big)\in\PP_{T\FF}(M)\times_{s,\pi_{\DS_{\g}(\pi_{B})}}\DS_{\g}(\pi_{B})
	\]
	defines a leafwise transport functor $T_{\pi_{B}^{\g}}:\PP_{T\FF}(M)\rightarrow\Aut(\pi_{\DS_{\g}(\pi_{B})})$.\qed
\end{thm}

\begin{rmk}\label{assoc2}\normalfont
	The functor $T_{\pi_{B}^{\g}}$ can also be constructed directly from $T_{\pi_{B}}$ of Theorem \ref{bicboi}and $T_{\FF}$ of Theorem \ref{Tf} using the associated bundle picture considered in Remark \ref{assoc}.  Regarding $\DS_{\g}(\pi_{B})$ as the space $(\DS_{\g}(M,\FF)\times_{M}\BS_{\g})/\g\Diff_{0}^{\loc}(\RB^{q})$, the leafwise transport functor $T_{\pi_{B}^{\g}}$ takes the form
	\[
	T_{\pi_{B}^{\g}}(\gamma,d)\big([((x,[f]_{x}),[\eta]_{0})]\big):=[(T_{\FF}(\gamma,d)(x,[f]_{x}),[T_{\pi_{B}}\circ\eta]_{0})]
	\]
	for all $(x,[f]_{x})\in\DS_{\g}(M,\FF)$ and $[\eta]_{0}\in\BS_{\g}$.
\end{rmk}

\begin{defn}
	Let $\pi_{B}:B\rightarrow M$ be a foliated bundle.  The groupoid $\HH(T_{\pi_{B}^{\g}})$ associated to the leafwise transport functor $T_{\pi_{B}^{\g}}:\PP_{T\FF}(M)\rightarrow\Aut(\pi_{\DS_{\g}(\pi_{B})})$ of Theorem \ref{TfB} is called the \textbf{holonomy groupoid of the foliated bundle $\pi_{B}:B\rightarrow M$}
\end{defn}

\begin{ex}\label{fr}\normalfont
	Let $(M,\FF)$ be a foliated manifold of codimension $q$.  The transverse frame bundle $\pi_{\Fr(M/\FF)}:\Fr(M/\FF)\rightarrow M$ associated to $(M,\FF)$, whose fibre over $x\in M$ consists of all linear isomorphisms $\phi:\RB^{q}\rightarrow T_{x}M/T_{x}\FF$, is a foliated bundle.  This can be seen using the foliated charts of $(M,\FF)$.
	
	The holonomy groupoid $\HH(M,\FF)$ of $(M,\FF)$ is \emph{equal} to the holonomy groupoid $\HH(T_{\pi_{\Fr(M/\FF)}^{\g}})$ of the foliated transverse frame bundle.  Indeed, regard $\DS_{\g}(\pi_{\Fr(M/\FF)})$ as the associated bundle
	\[
	\big(\DS_{\g}(M,\FF)\times_{M}\BS_{\g}\big)/\g\Diff_{0}^{\loc}(\RB^{q}),
	\]
	where $\BS_{\g}$ is the space of germs at $0$ of locally defined smooth functions $\RB^{q}\rightarrow\Fr(M/\FF)$.  Now as in Remark \ref{assoc2}, the leafwise transport functor $T_{\pi_{\Fr(M/\FF)}^{\g}}$ is simply given by $T_{\FF}$ on the first coordinate and $T_{\pi_{\Fr(M/\FF)}^{0}}$ on the second.  Two paths being equivalent under $T_{\FF}$ implies that they are also equivalent under $T_{\pi_{\Fr(M/\FF)}^{0}}$, from which it follows that $\HH(M,\FF) = \HH(T_{\pi_{\Fr(M/\FF)}^{\g}})$.
\end{ex}

\subsection{Jets and a hierarchy of holonomy groupoids}

Let us continue to consider a foliated bundle $\pi_{B}:B\rightarrow M$, associated to a codimension $q$ foliation $\FF$ on $M$.  Recall that there is a projective system $\pi_{B}^{l,k}:J^{k}_{\t}(\pi_{B})\rightarrow J^{l}_{\t}(\pi_{B})$ of transverse jet bundles $\pi_{B}^{k}:J^{k}_{\t}(\pi_{B})\rightarrow M$ associated to $\pi_{B}$, where the manifold $J^{k}_{\t}(\pi_{B})$ consists of classes of distinguished sections of $\pi_{B}$ with the same $k$-jet (see Remark \ref{jets}).  We have the following relationship between $\DS_{\g}(\pi_{B})$ and the jet bundles $J^{k}_{\t}(\pi_{B})$.

\begin{prop}\label{gjet}
	For each $k\in\NB$, the projection $\pi_{B}^{k,\g}:\DS_{\g}(\pi_{B})\rightarrow J^{k}_{\t}(\pi_{B})$ defined by
	\[
	\pi_{B}^{k,\g}\big((x,[\sigma]_{x})\big):=j^{k}_{x}\sigma
	\]
	is smooth, and for any $l<k$ one has $\pi_{B}^{l,\g} = \pi_{B}^{l,k}\circ\pi_{B}^{k,\g}$.
\end{prop}

\begin{proof}
	That $\pi_{B}^{l,\g} = \pi_{B}^{l,k}\circ\pi_{B}^{k,\g}$ is true follows from the definition.  Let us therefore check smoothness of $\pi_{B}^{k,\g}$.  Let $\rho:U\rightarrow\DS_{\g}(\pi_{B})$ be a plot, which we may assume without loss of generality to be of the form
	\[
	\rho(u) = \big(\tilde{x}(u),[\tilde{\sigma}(u)]_{\tilde{x}(u)}\big),\hspace{7mm}u\in U,
	\]
	where $\tilde{x}$ and $\tilde{\sigma}$ are plots of $M$ and $\DS(\pi_{B})$ respectively.  We must check that $\pi_{B}^{k,\g}\circ\rho:U\rightarrow J^{k}_{\t}(\pi_{B})$ is smooth.  We have
	\[
	\pi_{B}^{k,\g}\circ\rho(u) = j^{k}_{\tilde{x}(u)}\tilde{\sigma}(u).
	\]
	Smoothness of $\tilde{\sigma}$ guarantees that all the partial derivatives of $\tilde{\sigma}(u)$, taken with respect to any foliated chart about $\tilde{x}(u)$ and any local trivialisation of $B$ over this chart, vary smoothly with respect to $u$.  It follows by definition of the coordinates on $J^{k}_{\t}(\pi_{B})$ (see Remark \ref{jets}) that $\pi_{B}^{k,\g}\circ\rho$ is smooth.
\end{proof}

As a consequence of Proposition \ref{gjet} and the universal property of the projective limit, the tower
\begin{center}
	\begin{tikzcd}[row sep = huge]
	 & & & J^{\infty}_{\t}(\pi_{B}) \ar[dl,"\pi_{B}^{k+1,\infty}"] \ar[dr,"\pi_{B}^{k,\infty}"'] \ar[drrr,"\pi_{B}^{0,\infty}"'] & & &\\ & \hdots\ar[r] & J^{k+1}_{\t}(\pi_{B}) \ar[rr,"\pi_{B}^{k,k+1}"'] & & J^{k}_{\t}(\pi_{B}) \ar[r] & \hdots \ar[r] & B 
	\end{tikzcd}
\end{center}
``completes" to a tower
\begin{center}
	\begin{tikzcd}[row sep = large]
	& & & \DS_{\g}(\pi_{B}) \ar[d,"\pi_{B}^{\infty,\g}"] & & & \\ & & & J^{\infty}_{\t}(\pi_{B}) \ar[dl,"\pi_{B}^{k+1,\infty}"] \ar[dr,"\pi_{B}^{k,\infty}"'] \ar[drrr,"\pi_{B}^{0,\infty}"'] & & &\\ & \hdots\ar[r] & J^{k+1}_{\t}(\pi_{B}) \ar[rr,"\pi_{B}^{k,k+1}"'] & & J^{k}_{\t}(\pi_{B}) \ar[r] & \hdots \ar[r] & B
	\end{tikzcd}
\end{center}
of bundles over $M$.  A similar fact is true for the holonomy groupoid $\HH(T_{\pi_{B}^{\g}})$ of $\pi_{B}$ and the fibre holonomy groupoids $\HH(T_{\pi_{B}^{k}})$ associated to the foliated jet bundles $\pi_{B}^{k}:J^{k}_{\t}(\pi_{B})\rightarrow M$.

\begin{thm}\label{hierarchy}
	Let $\pi_{B}:B\rightarrow M$ be a foliated bundle.  Then for each $k\in\NB$, there are surjective morphisms $\Pi_{B}^{k,\g}:\HH(T_{\pi_{B}^{\g}})\rightarrow\HH(T_{\pi_{B}^{k}})$ and $\Pi_{B}^{k,k+1}:\HH(T_{\pi_{B}^{k+1}})\rightarrow\HH(T_{\pi_{B}^{k}})$ of diffeological groupoids for which $\Pi_{B}^{k,\g} = \Pi_{B}^{k,k+1}\circ\Pi_{B}^{k+1,\g}$.  Consequently we have a tower
	\begin{center}
		\begin{tikzcd}[row sep = large]
		& & & \HH(T_{\pi_{B}^{\g}}) \ar[d,"\Pi_{B}^{\infty,\g}"] & & & \\ & & & \HH(T_{\pi_{B}^{\infty}}) \ar[dl,"\Pi_{B}^{k+1,\infty}"] \ar[dr,"\Pi_{B}^{k,\infty}"'] \ar[drrr,"\Pi_{B}^{0,\infty}"'] & & &\\ & \hdots\ar[r] & \HH(T_{\pi_{B}^{k+1}}) \ar[rr,"\Pi_{B}^{k,k+1}"'] & & \HH(T_{\pi_{B}^{k}}) \ar[r] & \hdots \ar[r] & \HH(T_{\pi_{B}})
		\end{tikzcd}
	\end{center}
	of diffeological groupoids, which we refer to as the \textbf{hierarchy of holonomy groupoids} for the foliated bundle $\pi_{B}$.
\end{thm}

\begin{proof}
	Suppose that elements $(\gamma_{1},d_{1})$ and $(\gamma_{2},d_{2})$ of $\PP_{T\FF}(M)$ are contained in the same fibre of $T_{\pi_{B}^{\g}}:\PP_{T\FF}(M)\rightarrow\Aut(\pi_{\DS_{\g}(\pi_{B})})$.  Then they must also be contained in the same fibre of $T_{\pi_{B}^{k}}:\PP_{T\FF}(M)\rightarrow\Aut(\pi_{B}^{k})$ for any $k$.  Indeed, consider the liftings $L^{\g}_{B}:\PP_{T\FF}(M)\times_{s,\pi_{\DS_{\g}(M,\FF)}}\DS_{\g}(M,\FF)\rightarrow \PP_{H^{\g}_{B}}(\DS_{\g}(\pi_{B}))$ and $L^{k}_{B}:\PP_{T\FF}(M)\times_{s,\pi_{B}^{k}}J^{k}_{\t}(\pi_{B})\rightarrow\PP_{H^{k}_{B}}(J^{k}_{\t}(\pi_{B}))$ of Theorem \ref{liftingB} and Remark \ref{jets2} respectively.  Then the projection $\pi_{B}^{k,\g}:\DS_{\g}(\pi_{B})\rightarrow J^{k}_{\t}(\pi_{B})$ induces a surjective functor $\PP\pi_{B}^{k,\g}:\PP(\DS_{\g}(\pi_{B}))\rightarrow\PP(J^{k}_{\t}(\pi_{B}))$, which in particular sends $\PP_{H^{\g}_{B}}(\DS_{\g}(\pi_{B}))$ onto $\PP_{H^{k}_{B}}(J^{k}_{\t}(\pi_{B}))$ by Proposition \ref{gjet}.  By Theorem \ref{TfB}, the diagram
	\begin{center}
		\begin{tikzcd}
		\PP_{T\FF}(M)\times_{s,\pi_{\DS_{\g}(\pi_{B})}}\DS_{\g}(\pi_{B}) \ar[r, "L_{B}^{\g}"] \ar[d, "\id\times\pi_{B}^{k,\g}"] & \PP_{H^{\g}_{B}}(\DS_{\g}(\pi_{B})) \ar[d,"\PP\pi_{B}^{k,\g}"] \\ \PP_{T\FF}(M)\times_{s,\pi_{B}^{k}}J^{k}_{\t}(\pi_{B}) \ar[r, "L_{B}^{k}"] & \PP_{H^{k}_{B}}(J^{k}_{\t}(\pi_{B}))
		\end{tikzcd}
	\end{center}
	commutes.  Thus, letting $r(\gamma,d):=\gamma(d)$ denote the range map on any path space $\PP(X)$, we have $T_{\pi_{B}^{\g}}(\gamma_{1},d_{1}) = T_{\pi_{B}^{\g}}(\gamma_{2},d_{2})$ if and only if the equality 
	\[
	r\circ L_{B}^{\g}((\gamma_{1},d_{1}),(x,[\sigma]_{x}))(d_{1}) = r\circ L_{B}^{\g}((\gamma_{2},d_{2}),(x,[\sigma]_{x}))(d_{2})
	\] 
	holds for all $(x,[\sigma]_{x})\in\DS_{\g}(\pi_{B})$, which implies the equality
	\[
	r\circ L_{B}^{\g}((\gamma_{1},d_{1}),j^{k}_{x}\sigma)(d_{1}) = r\circ L_{B}^{\g}((\gamma_{2},d_{2}),j^{k}_{x}\sigma)(d_{2}),
	\]
	for all $j^{k}_{x}\sigma\in J^{k}_{\t}(\pi_{B})_{x}$.  Hence $T_{\pi_{B}^{k}}(\gamma_{1},d_{1}) = T_{\pi_{B}^{k}}(\gamma_{2},d_{2})$.  Thus the surjective map
	\[
	\Pi_{B}^{k,\g}:\HH(T_{\pi_{B}^{\g}})\ni[(\gamma,d)]_{\g}\mapsto[(\gamma,d)]_{k}\in\HH(T_{\pi_{B}^{k}})
	\]
	is a well-defined groupoid morphism which is, by definition of the quotient diffeology, smooth.  A similar argument shows that
	\[
	\Pi_{B}^{k,k+1}:\HH(T_{\pi_{B}^{k+1}})\ni[(\gamma,d)]_{k+1}\mapsto[(\gamma,d)]_{k}\in\HH(T_{\pi_{B}^{k}})
	\]
	is a well-defined morphism of diffeological groupoids, and the equation $\Pi_{B}^{k,\g} = \Pi_{B}^{k,k+1}\circ\Pi_{B}^{k+1,\g}$ is clear.  Finally the commuting diagram given in the statement follows from the universal property of the projective limit.
\end{proof}

	\bibliographystyle{amsplain}
	\bibliography{references}
	
\end{document}